\documentclass[a4paper,11pt]{article}

\usepackage{amssymb}
\usepackage{latexsym}
\usepackage{amsmath,amsfonts,bm}
\usepackage{mathrsfs}
\usepackage{enumerate, geometry}
\usepackage{graphicx,psfrag}
\usepackage{amsthm}
\usepackage{color}

\usepackage{pstricks}

\usepackage{xspace}

\usepackage{color}

\newcommand{\note}[1]%
{\noindent\centerline{\fbox{\parbox{.9\textwidth}{\textbf{#1}}}}}
\newcommand{\snote}[1]%
{\fbox{\textbf{#1}}}

\def\proof{\par{\noindent\it Proof}. \ignorespaces}
\def\endproof{\vbox{\hrule height0.6pt\hbox{%
   \vrule height1.3ex width0.6pt\hskip0.8ex
   \vrule width0.6pt}\hrule height0.6pt}}

\newcommand{\bnew}[1]{{\color{black}{#1}}}
\newcommand{\rnew}[1]{{\color{black}{#1}}}

\input definitions.sty

\begin{document}

\title{Contraction and optimality properties\\ of adaptive Legendre-Galerkin methods:\\ the $1$-dimensional case}

\author{Claudio Canuto$^a$, Ricardo H. Nochetto$^b$ and Marco Verani$^c$}

\date{June 24, 2012}

\maketitle 

%%%%%%%%%%%%%% ADDRESSES %%%%%%%%%%%
\begin{center}
{\small
$^a$ Dipartimento di Scienze Matematiche, Politecnico di Torino\\
        Corso Duca degli Abruzzi 24, 10129 Torino, Italy\\
	E-mail: {\tt claudio.canuto@polito.it}\\
\vskip 0.1cm
%%%%%%%%%%%%%%%%%%%%%%%%%%%%%%%%
Department of Mathematics and Institute for Physical Science
  and Technology,\\ University of Maryland, College Park, MD 20742, USAy\\
E-mail: {\tt rhn@math.umd.edu} 
\vskip 0.1cm
%%%%%%%%%%%%%%%%%%%%%%%%%%%%%%%%%
$^c$ MOX, Dipartimento di Matematica, Politecnico di Milano\\
Piazza Leonardo da Vinci 32, I-20133 Milano, Italy\\
E-mail: {\tt marco.verani@polimi.it}\\
}
\end{center}

\begin{abstract}
\noindent As a first step towards a mathematically rigorous understanding of
\rnew{adaptive spectral/$hp$} discretizations of elliptic boundary-value
problems, we study the performance of adaptive Legendre-Galerkin methods
in one space dimension. These methods offer unlimited approximation power
only restricted by solution and data regularity. Our investigation is
inspired by a similar study that we recently \rnew{carried out} for
Fourier-Galerkin methods in a periodic box. We first consider an ``ideal''
algorithm, which we prove to be convergent at a fixed rate. Next we
enhance its performance, \rnew{consistently} with the expected fast error decay of
high-order methods, by activating a larger set of degrees of freedom at
each iteration. We guarantee optimality (in the non-linear approximation
sense) by incorporating a coarsening step. Optimality is measured in terms
of certain sparsity classes of the Gevrey type, which describe a
(sub-)exponential decay of the best approximation error.

\end{abstract}

%%%%%%%%%%%%%%%%%%%%%%%%%%%%%%%%%%%%%%%%%%%%%
\section{Introduction}

The mathematical theory of adaptive algorithms for approximating the solution of multidimensional 
elliptic PDEs is rather recent. The first convergence
results of adaptive finite element methods (AFEM) have been proved by D\"orfler \cite{dorfler:96} and Morin, Nochetto, and Siebert \cite{MNS:00}.
On  the other hand, the first convergence rates were derived by Cohen, Dahmen, and DeVore
\cite{CDDV:1998} for wavelets in any dimensions $d$, and for finite element
methods (AFEM)
by Binev, Dahmen, and DeVore \cite{BDD:04} for
$d=2$ and Stevenson \cite{Stevenson:2007} for any $d$. 
The most comprehensive results
for AFEM are those of Casc\'on, Kreuzer, Nochetto,
and Siebert \cite{Nochetto-et-al:2008} for any $d$ and $L^2$ data,
and Cohen, DeVore, and Nochetto \cite{CDN:11} for $d=2$ and $H^{-1}$
data. The key result of this theory is that   
AFEM delivers a convergence rate compatible with that 
of the approximation classes where the solution and data 
belong. The recent results in \cite{CDN:11} reveal that it is the
approximation class of the solution that really matters. However, 
in all the above cases (\rnew{practical} wavelets and FEM)
the convergence rates are limited by the approximation power of
the method, which is finite and related to the
polynomial degree of the basis functions or the number of their vanishing moments, as well as the 
regularity of the solution and data.
\rnew{We refer to the surveys \cite{NSV:09} by Nochetto, Siebert and Veeser
for AFEM and \cite{Stevenson:09} by Stevenson for adaptive wavelets.}

A similar study for adaptive high-order methods (such as spectral element methods or $hp$ finite element methods) has started only recently and very few results are available. The ultimate goal 
is to design algorithms which {\sl optimally} choose between $h$-refinement and 
$p$-enrichment, and for which a rigorous mathematical proof of rate of
convergence and optimality can be given. This is a formidable task
which requires, \rnew{among other things}, the study of sparsity
classes tailored to functions which are (locally) infinitely
smooth. Some rigorous mathematical results on the convergence of
$hp$-methods for PDEs have been established 
\rnew{recently} in \cite{Schmidt-Siebert:00,Doerfler-Heuveline:07, Burg-Doerfler:11}. Although the numerical implementations of adaptive $hp$ methods has now reached high levels of efficiency (see, e.g., \cite{Rachowicz-etal:06}), the theoretical study of optimality has never been addressed. A first step in this direction has been accomplished in \cite{CNV:2011}, where the contraction and the optimal cardinality properties of 
adaptive Fourier-Galerkin methods in a
periodic box in $\mathbb{R}^d$ with dimension $d\ge1$ are presented 
together with the analysis of suitable  nonlinear approximation classes 
(the classical one corresponding \rnew{to} algebraic decay of the Fourier coefficients and another one associated \rnew{with} exponential decay).

The present paper represents a second step towards the study of optimality for high-order methods. We extend the method and the results contained in \cite{CNV:2011} to a non-periodic setting in one dimension. This is the closest situation to the periodic case, since an orthonormal basis is readily available (the so called Babu\v ska-Shen basis formed by the anti-derivatives of the Legendre polynomials); together with the associated dual basis, it allows one to represent the norm of a function or a functional as a $\ell_2$-type norm of the vector of its coefficients. In addition, the stiffness matrix for smooth coefficients of the differential operator exhibits a quasi-sparsity behavior, i.e., an exponential decay of its entries as one goes away from the diagonal.

In this paper we only consider the case of an exponential decay of the best approximation error of the solution of the PDE; indeed this is the most relevant situation which motivates the use of a spectral/$p$ method. Our approach relies on a careful analysis of the relation between the sparsity class of a function and the sparsity class of its image through the differential operator. As already pointed out in the analysis of the Fourier method \cite{CNV:2011}, the discrepancy between the sparsity classes of the residual and the exact solution suggests the introduction of a coarsening step that guarantees the optimality of the computed approximation at the end of each adaptive iteration.

The multi-dimensional situation, which poses additional difficulties, is currently under investigation.

%%%%%%%%%%%%%%%%%%%%%%%%%%%%%%%%%%%%%%%%%%%
\section{Legendre and Babu\v ska-Shen bases}
%%%%%%%%%%%%%%%%%%%%%%%%%%%%%%%%%%%%%%%%%%%%

Let $I=(-1,1)$ and $L_k({x})$, $k \geq 0$, be the $k$-th Legendre orthogonal polynomial in $I$, which satisfies  ${\rm deg}\, L_k = k$, $L_k(1)=1$ and
$$
\int_I L_k({x}) L_m({x}) \, d{x} = \frac2{2k+1}\, \delta_{km}\;, \qquad m \geq 0 \;.
$$
Furthermore, we denote by 
$$
\phi_k(x)= \sqrt{k+1/2}\  L_k(x)\;, \qquad k \geq 0 \;,
$$ 
the elements of the orthonormal {\sl Legendre basis} in $L^2(I)$, which satisfy 
$$ 
\int_I \phi_k (x) \phi_m(x) dx =\delta_{km}\ ,\quad m\geq 0 .
$$
We denote by $D = d/dx$ the first derivative operator. 
The natural modal basis in $H^1_0(I)$ is the {\sl Babu\v ska-Shen basis} (BS basis), whose elements are defined as
\begin{equation}\label{eq:defBS}
\eta_k({x})=\sqrt{k-1/2}\int_{{x}}^1 L_{k-1}(s)\,{d}s =
\frac1{\sqrt{4k-2}}\big(L_{k-2}({x})-L_{k}({x})\big) 
 \qquad k \geq 2\;;
\end{equation}
they satisfies ${\rm deg}\, \eta_k = k$ and
\begin{equation}\label{eq:propBS.05}
D \eta_k =-\phi_{k-1} \;.
\end{equation}
Thus, the $\eta_k$'s satisfy
\begin{equation}\label{eq:propBS.1}
(\eta_k,\eta_m)_{H^1_0({I})} = 
\int_I D \eta_k({x}) D \eta_m({x}) \, d{x} = \delta_{km}\;, \qquad k,m \geq 2 \;,
\end{equation}
i.e., they form an orthonormal basis for the ${H^1_0({I})}$-inner product.

Equivalently, the (semi-infinite) stiffness matrix $\bm{S}_\eta$ of the Babu\v ska-Shen basis with respect to this inner product is
the identity matrix $\bm{I}$. On the other hand, one has
\begin{equation}\label{eq:propBS.2}
(\eta_k,\eta_m)_{L^2({I})} = 
\begin{cases}
\frac2{(2k-3)(2k+1)} & \text{if } m=k \;, \\
- \frac1{(2k+1)\sqrt{(2k-1)(2k+3)}} &  \text{if } m=k+2 \;, \\
0 & \text{elsewhere.}
\end{cases}  \qquad \text{for } k \geq m \;, 
\end{equation}
which means that the mass matrix $\bm{M}_\eta$ is pentadiagonal. (Since even and odd modes are mutually orthogonal,
the mass matrix could be equivalently represented by a couple of tridiagonal matrices, each one collecting the inner
products of all modes with equal parity).  
For every $v\in H^1_0(I)$ we have $$ v(x)=\sum_{k=2}^\infty \hat{v}_k \eta_k(x)$$
with $\hat{v}_k=\int_{-1}^1 D v(x)D\eta_k(x) dx$. In view of the results in the next sections, we  
observe that (\ref{eq:propBS.05}) yields
\begin{equation}
Dv  =\sum_{k=2}^\infty \hat{v}_k D \eta_k = 
- \sum_{k=2}^\infty \hat{v}_k \phi_{k-1} \;; 
\end{equation}
comparing this expression with 
$$
Dv=\sum_{h=1}^\infty {(Dv)}^{\wedge}_h \phi_h 
$$ 
yields
\begin{equation}\label{aux:coeff}
(Dv)^\wedge_h=-\hat{v}_{h+1} \qquad\forall h\geq 1 \;.
\end{equation}

From \eqref{eq:propBS.1}, there follows that the $H^1_0({I})$-norm can be expressed, according to the Parseval identity, as
\begin{equation}\label{eq:propBS.3}
\Vert v \Vert_{H^1_0({I})}^2  =  \sum_{k= 2}^\infty |\hat{v}_k|^2 = \bm{v}^T \bm{v}\;,
\end{equation}
where the vector $\bm{v}=(\hat{v}_k)$ collects the coefficients of $v$. The $L^2(I)$-norm of $v$ is
given by
\begin{equation}\label{eq:propBS.3bis}
\Vert v \Vert_{L^2({I})}^2  = \bm{v}^T  \bm{M}_\eta \, \bm{v}\;.
\end{equation}
Correspondingly, any element $f \in H^{-1}({I})$ can be expanded 
\bnew{in terms of} the {\sl dual Babu\v ska-Shen basis}, whose elements
$\eta_k^*$, $k \geq 2$, are defined by the conditions
$$ 
\langle \eta_k^* , v \rangle = \hat{v}_k \qquad \forall v \in H^1_0({I}) \;;
$$ 
precisely one has
$$
f = \sum_{k = 2}^\infty \hat{f}_k \eta_k^* \;, \qquad  \text{with} \ \ \hat{f}_k = \langle f,\eta_k \rangle \;,
$$
and its $H^{-1}({I})$-norm can be expressed, according to the Parseval identity, as
\begin{equation}\label{eq:propBS.4}
\Vert f \Vert_{H^{-1}({I})}^2  =  \sum_{k= 2}^\infty |\hat{f}_k|^2 \;.
\end{equation}

Summarizing, we see that the one-dimensional Legendre case is similar, 
from the point of view of 
expansions and norm representations, to the Fourier case (see \cite{CNV:2011}).

Throughout the paper, we will use the notation $\Vert \ . \ \Vert$ to
indicate both the $H^1_0(I)$-norm 
of a function $v$, or the $H^{-1}(I)$-norm of a linear
form $f$; the specific meaning will be clear from the context.

Moreover, given any finite index set $\Lambda \subset {\mathbb{N}_2}:=\{k\in\mathbb{N}:~k\geq 2\}$, we define
the subspace of $V:=H^1_0(I)$
$$
V_{\Lambda} := {\rm span}\,\{\eta_k\, | \, k \in \Lambda \}\;;
$$
we set $|\Lambda|= \rm{card}\, \Lambda$, so that $\rm{dim}\, V_{\Lambda}=|\Lambda|$. If $g$ admits an
expansion $g = \sum_{k=2}^\infty \hat{g}_k \eta_k $ (converging in an appropriate norm), then we define its 
projection $P_\Lambda g$ \bnew{onto} $V_\Lambda$ by setting
$$
P_\Lambda g = \sum_{k \in \Lambda} \hat{g}_k \eta_k \;.
$$

%%%%%%%%%%%%%%%%%%%%%%%%%%%%%%%%%%%%%%%%%%%
\section{The model problem and its Galerkin discretization}

We now consider the elliptic problem
\begin{equation}\label{eq:four03}
\begin{cases} 
Lu=-D \cdot (\nu D u)+ \sigma u = f & \text{in } I \;, \\
u(-1) = u(1) = 0  \;, 
\end{cases} 
\end{equation}
where $\nu$ and $\sigma$ are sufficiently smooth real coefficients satisfying 
$0 < \nu_* \leq \nu(x) \leq \nu^* < \infty$ and $0 \leq \sigma(x) \leq \sigma^* < \infty$
in $I$; let us set
$$
\alpha_* = \nu_* \qquad \text{and} \qquad \alpha^* = \max(\nu^*, \sigma^*) \;.
$$
We formulate this problem variationally as
\begin{equation}\label{eq:four.1}
u \in H^1_0(I) \ \ : \quad a(u,v)= \langle f,v \rangle \qquad \forall v \in  H^1_0(I) \;,
\end{equation}
where $a(u,v)=\int_I \nu D u D{v} + \int_I\sigma u {v}$. We denote by 
$\tvert v \tvert = \sqrt{a(v,v)}$
the energy norm of any $v \in H^1_0(I)$, which satisfies 
\begin{equation}\label{eq:four.1bis}
\sqrt{\alpha_*}  \Vert v \Vert  \leq \tvert v \tvert \leq 
\sqrt{\alpha^*}  \Vert v \Vert \;.
\end{equation}

Given any finite set $\Lambda \subset \mathbb{N}_2$, 
the Galerkin approximation is defined as
\begin{equation}\label{eq:four.2}
u_\Lambda \in V_\Lambda \ \ : \quad a(u_\Lambda,v_\Lambda)= 
\langle f,v_\Lambda \rangle \qquad \forall v_\Lambda \in V_\Lambda \;.
\end{equation}
For any $w \in V_\Lambda$, we define the residual
$$
r(w)=f-Lw = \sum_k \hat{r}_k(w) \bnew{\eta_k^*} \;, \qquad \text{where} \qquad 
\hat{r}_k(w) = \langle f - Lw, \eta_k \rangle = \langle f,\eta_k \rangle -a(w,\eta_k) \;.
$$
Then, the previous definition of $u_\Lambda$ is equivalent to the condition
\begin{equation}\label{eq:four.2.1ter}
P_\Lambda r(u_\Lambda) = 0 \;, \qquad \text{i.e., } 
\quad \hat{r}_k(u_\Lambda)=0 \qquad \forall k \in \Lambda \;.
\end{equation}
On the other hand, by the continuity and coercivity of the bilinear
form $a$, one has 
\begin{equation}\label{eq:four.2.1}
\frac1{\alpha^*} \Vert r(u_\Lambda) \Vert \leq
\Vert u - u_\Lambda \Vert \leq 
\frac1{\alpha_*} \Vert r(u_\Lambda) \Vert \;,
\end{equation}
or, equivalently,
\begin{equation}\label{eq:four.2.1bis}
\frac1{\sqrt{\alpha^*}} \Vert r(u_\Lambda) \Vert \leq
\tvert u - u_\Lambda \tvert \leq 
\frac1{\sqrt{\alpha_*}} \Vert r(u_\Lambda) \Vert \;.
\end{equation}

%%%%%%%%%%%%%%%%%%%%%%%%%%%%%%%%%
\subsection{Algebraic representation and properties of the stiffness matrix}\label{sec:algebraic_repres}

Let us identify the solution $u = \sum_k \hat{u}_k \eta_k$ of Problem (\ref{eq:four.1})
with the vector $\mathbf{u}=(\hat{u}_k)$ 
of its Babu\v ska-Shen (BS) coefficients. Similarly, let us identify
the right-hand side $f$ with the vector $\mathbf{f}=(\hat{f}_\ell)$ of its BS coefficients.
Finally, let us introduce the semi-infinite, symmetric and positive-definite matrix 
\begin{equation}\label{eq:four100}
\mathbf{A}=(a_{\ell,k}) \qquad \text{with} \qquad 
a_{\ell,k}= a(\eta_k,\eta_\ell)_{k, \ell \geq 2} \;.
\end{equation}
Then, Problem (\ref{eq:four.1}) can be equivalently written as
\begin{equation}\label{eq:four110}
\mathbf{A} \mathbf{u} = \mathbf{f} \;. 
\end{equation}
In order to study the properties of the matrix ${\mathbf A}$, let us
first recall \bnew{the following result}. 
\begin{property}[\bnew{product of Legendre polynomials}]
There holds
\begin{equation}\label{product}
L_m(x)L_n(x)=\sum_{r=0}^{\min (m,n)} A_{m,n}^r L_{m+n-2r} (x)
\end{equation}
with 
$$ A_{m,n}^r:= \frac{A_{m-r} A_r A_{n-r}}{A_{n+m-r}}\frac{2n+2m-4r+1}{2n+2m-2r+1}$$
and $$ A_0:=1\ , \qquad A_m:=\frac{1\cdot 3 \cdot 5 \ldots (2m-1)}{m !}=\frac{(2m)!}{2^m (m!)^2} \ .$$
\end{property}
\begin{proof}
See e.g. \cite{Adams:1878}.
\end{proof}

Bearing in mind \eqref{eq:four100}, in the sequel we will make use of the following notation   $$ a_{m,n}=a_{m,n}^{(1)}+a_{m,n}^{(0)}, \quad m,n\geq 2$$
with $$a_{m,n}^{(1)}:=\int_{-1}^1 \nu(x) D \eta_{m}(x) D \eta_{n}(x) dx \quad\mathrm{and}\quad 
 a_{m,n}^{(0)}:=\int_{-1}^1 \sigma(x)\eta_{m}(x) \eta_{n}(x) dx .$$
\bnew{If} $\nu(x)=\sum_{k=0}^\infty \nu_k L_k$, then using \eqref{product} it is possible to prove 
that
\begin{eqnarray}\label{matrix:1}
a_{m+1,n+1}^{(1)}&:=&\int_{-1}^1 \nu(x) D \eta_{m+1}(x) D \eta_{n+1}(x) dx\nonumber\\ &=&\frac{\sqrt{(2m+1)(2n+1)}}{2}\int_{-1}^1\nu(x) L_m(x)L_n(x)dx\nonumber\\
&=&\sum_{r=0}^{\min(m,n)} B_{m,n}^r \nu_{m+n-2r}\ ,
\end{eqnarray}
where we set 
\begin{equation}\label{def:B}
B_{m,n}^r:=\frac{ \sqrt{(2m+1)(2n+1)}}{2m+2n-4r+1}A_{m,n}^r\ .
\end{equation}
%$ \frac{1}{\sqrt{\pi}}\frac{1}{\sqrt{m+n}}$.
%Let us study the behavior of $\vert a_{m,n}^{(1)}\vert $ by assuming a certain exponential decay of the Legendre coefficients $\nu_k$, e.g. 
%$\vert \nu_k\vert\   \leq C_\eta e^{-\eta k}$, for a certain $\eta>0$.
\begin{property}[\bnew{coefficients $B_{m,n}^r$}]\label{property:B}
There holds $$ B_{m,n}^r\lesssim 1
$$ for every $m,n\geq 0$ and $0\leq r \leq \min(m,n)$. 
\end{property}
\begin{proof}
We \bnew{first observe that Stirling's formula $ m! \sim \sqrt{2\pi m} e^{-m} m^m$
implies}
\begin{equation}
A_m\sim \frac{2^m}{\sqrt{\pi m}}\ .
\end{equation}
\bnew{This} can be used to prove asymptotic estimates
\rnew{for the factor $\frac{A_{m-r} A_r A_{n-r}}{A_{n+m-r}}$ of $A^r_{m,n}$:}
\begin{enumerate}[\rnew{$\bullet$}]
\item Case $0< r <\min(m,n)$:
\begin{equation*}\label{asympt:1}
 \frac{A_{m-r} A_r A_{n-r}}{A_{n+m-r}}\sim 
 \frac{1}{\pi} \frac{\sqrt{n+m-r}}{\sqrt{m-r}\sqrt{n-r}\sqrt{r}}\ ;
\end{equation*}
\item Case $r=0$:
\begin{equation*}\label{asympt:2}
 \frac{A_{m} A_{n}}{A_{n+m}}\sim \frac{1}{\sqrt{\pi}} \frac{\sqrt{n+m}}{\sqrt{nm}} \ ;
\end{equation*}
\item Case $r=\min(m,n)$ \rnew{and $m\not= n$}:
\begin{equation*}\label{asympt:3}
\frac{A_{\min(m,n)} A_{\vert m -n \vert}}{{\rnew{A_{\max(m,n)}}}}\sim \frac{1}{\sqrt{\pi}} \frac{\sqrt{\max(m,n)}}{\sqrt{\min(m,n)}\sqrt{\vert m -n \vert}}\ . 
\end{equation*}
\rnew{When $m=n$ it is sufficient to use $A_0=1$ to get $\frac{A_{m} A_{0}}{{\rnew{A_{m}}}}=1$.
}
\end{enumerate}
\rnew{Combining these results with \eqref{def:B},
we now estimate the quantities $B_{m,n}^r$ as follows:}
\begin{enumerate}[(a)]
\item Case $0 < r < \min(m,n)$:  
\begin{eqnarray}\label{asympt:1-1}
B_{m,n}^r&\sim&  
\frac{1}{\pi} \frac{ \sqrt{(2m+1)(2n+1)}}{2m+2n-2r+1}\frac{\sqrt{n+m-r}}{\sqrt{(m-r)(n-r)r}}
\nonumber\\
&\sim& \frac{1}{\pi} \frac{ \sqrt{mn}}{\sqrt{m+n-r}}\frac{1}{\sqrt{m-r}\sqrt{n-r}\sqrt{r}}\nonumber \ .
\end{eqnarray} 
We note that the above last term can be asymptotically bounded by 
\begin{equation*}
\rnew{B_{m,n}^r ~\sim~} \frac{1}{\pi}\frac{1}{\sqrt{\min(m+n,\vert m-n\vert)}}\ ,
\end{equation*}
which amounts to \rnew{considering} the extremal cases $r=1$ and $r=\min(m,n)-1$.
\item Case $r=0$:  
\begin{equation*}
B_{m,n}^0\sim \frac{1}{\sqrt{\pi}} \frac{1}{\sqrt{m+n}}\ .
\end{equation*}
\item Case $r=\min(m,n)$ and \rnew{$m\not=n$}: 
\begin{equation*}
B_{m,n}^{\min(m,n)}\sim \frac{1}{\sqrt{\pi}} \frac{1}{\sqrt{\vert n-m \vert}}\ .
\end{equation*}
\rnew{When $m=n$ we obtain $B_{m,m}^m\sim 1$.}
\end{enumerate}
Hence, by combining (a)-(c), there follows that the terms $B_{m,n}^r$, $0 \leq r \leq \min(m,n)$ 
are asymptotically bounded by a constant independent of $m,n,r$.
\end{proof}
\begin{proposition}[\bnew{decay of $a_{m,n}^{(1)}$}]\label{prop:stiff-1}
If there exists $\eta>0$ and a positive constant $C_\eta$ only depending on $\eta$ such that 
$$\vert \nu_k\vert\ \leq C_\eta e^{-\eta k}\qquad \forall k\geq 0 \ ,$$ 
then there holds
\begin{equation}\label{eq:stiff-1}
 \vert a_{m,n}^{(1)} \vert \leq C e^{-\eta\vert n-m\vert} \qquad \forall n,m\geq 2\ ,
\end{equation}
where $C$ is a constant only depending on $\eta$.
\end{proposition}
\begin{proof}
Using \eqref{matrix:1} and Property \ref{property:B} there follows 
\begin{eqnarray}
 \vert a_{m,n}^{(1)} \vert &\lesssim& \sum_{r=0}^{\min(m,n)} \vert \nu_{ m+n-2r} \vert 
 ~\lesssim \sum_{r=0}^{\min(m,n)} e^{-\eta(m+n-2r)}\nonumber\\
 &\lesssim& e^{-\eta\vert m-n\vert} \sum_{r=0}^{\min(m,n)}e^{-2\eta (\min(m,n)-r)}
~\lesssim C_\eta e^{-\eta \vert m-n\vert}\ .\nonumber
\end{eqnarray}
\rnew{This gives \eqref{eq:stiff-1} as asserted.}
\end{proof}
%---------------- reaction term -------------

Let $\sigma(x)=\sum_{k=0}^\infty \sigma_k L_k$ then using \eqref{eq:defBS} and\eqref{product} it is possible to prove that
\begin{eqnarray}\label{matrix:2}
a_{m,n}^{(0)}&:=&\int_{-1}^1 \sigma(x) \eta_{m}(x) \eta_{n}(x) dx\nonumber\\
&=& \frac{1}{\sqrt{(2m-1)(2n-1)}}\int_{-1}^1\sigma(x) (L_{m-2}(x)-L_m(x))(L_{n-2}(x)-L_n(x))dx\nonumber\\
&=&\frac{1}{\sqrt{(2m-1)(2n-1)}} \left\{
\sum_{r=0}^{\min(m-2,n-2)} C_{m-2,n-2}^r \sigma_{m-2+n-2-2r}\right .\nonumber\\
&&+ \sum_{r=0}^{\min(m-2,n)} C_{m-2,n}^r \sigma_{m-2+n-2r}+ \sum_{r=0}^{\min(m,n-2)} C_{m,n-2}^r \sigma_{m+n-2-2r}\nonumber\\
&& \left .+ \sum_{r=0}^{\min(m,n)} C_{m,n}^r \sigma_{m+n-2r}
\right\}\ ,
\end{eqnarray}
where 
$$ C_{m,n}^r:=  \frac{A_{m,n}^r }{2m+2n-4r+1}\ .$$
\begin{property}[\bnew{coefficients $C_{m,n}^r$}]\label{property:C}
There holds 
$$ \frac{1}{\sqrt{(2m-1)(2n-1)}}{C_{m-k,n-j}^r} \lesssim 1
$$ with $k,j=0,2$ and $m,n\geq 0$ and $0\leq r \leq \min(m-k,n-j)$.
\end{property}
\begin{proof}
Let us first consider the case $k,j=0$. Proceeding as in the proof of
Property \ref{property:B} we get the following cases \rnew{for the
auxiliary quantity
$D_{m,n}^r = \frac{1}{\sqrt{(2m-1)(2n-1)}}C_{m,n}^r$:}
\begin{enumerate}[(a)]
\item Case \rnew{$0 < r < \min(m,n)$:}
\begin{eqnarray}
\rnew{D_{m,n}^r}
\sim\frac{1}{\pi} \frac{1}{\sqrt{(2m-1)(2n-1)}}\frac{1}{2m+2n-2r+1}\frac{\sqrt{n+m-r}}{\sqrt{(m-r)(n-r)r}}\ .\nonumber
\end{eqnarray}
\item Case  $r=0$:
\begin{eqnarray}
\rnew{D_{m,n}^r} \sim\frac{1}{\sqrt{\pi}} \frac{1}{\sqrt{(2m-1)(2n-1)}}\frac{1}{2m+2n+1}\frac{\sqrt{n+m}}{\sqrt{mn}}\ .\nonumber
\end{eqnarray}
\item Case $r=\min(m,n)$:
\begin{eqnarray}
\rnew{D_{m,n}^r} \sim
\frac{1}{\sqrt{\pi}} \frac{1}{\sqrt{(2m-1)(2n-1)}}\frac{1}{2\max(m,n)+1}\frac{\sqrt{\max(n,m)}}{\rnew{\sqrt{\min(m,n)}\sqrt{\vert m-n\vert}}}\ .\nonumber
\end{eqnarray}
\end{enumerate}
We note that each right-hand side term in (a)-(c) is asymptotically bounded by a constant.
\rnew{Therefore the quantities $D_{m,n}^r$}, $0 \leq r \leq \min(m,n)$ are asymptotically bounded by a constant independent of $n,m,r$. The other terms in \eqref{matrix:2} can be treated similarly.
\end{proof} 
\begin{proposition}[\bnew{decay of $a_{m,n}^{(0)}$}]\label{prop:stiff-2}
If there exists $\eta>0$ and a positive constant $C_\eta$ only depending on $\eta$ such that 
$$\vert \sigma_k\vert\ \leq C_\eta e^{-\eta k}\qquad \forall k\geq 0\ ,$$ 
then there holds
\begin{equation}
 \vert a_{m,n}^{(0)} \vert \leq C e^{-\eta\vert n-m\vert} \qquad \forall n,m\geq 0\ ,
\end{equation}
where $C$ is a constant only depending on $\eta$.
\end{proposition}
\begin{proof}
Use \eqref{matrix:2} together with Property \ref{property:C}
and argue as in the proof of Proposition \ref{prop:stiff-1}.
\end{proof}

Combining Propositions \ref{prop:stiff-1} and \ref{prop:stiff-2} yields
\begin{corollary}[\bnew{decay of $a_{m,n}$}]
If there exists $\eta>0$ and a positive constant $C_\eta$ only depending on $\eta$ such that 
$$\vert \nu_k \vert , \  \vert \sigma_k \vert  \leq C_\eta e^{-\eta k}\qquad \forall k\geq 0\ ,$$ 
then there holds
\begin{equation}
 \vert a_{m,n} \vert \leq C e^{-\eta\vert n-m\vert} \qquad \forall n,m\geq 2\ ,
\end{equation}
where $C$ is a constant only depending on $\eta$.
\end{corollary}

Correspondingly, the matrix ${\bf A}$ belongs to the following
class.

\begin{definition}[{regularity classes for $\bA$}]\label{def:class.matrix}
A matrix ${\bf A}$ is said to belong to the {exponential} class 
${\mathcal D}_e(\eta_L)$ if there exists a constant $c_L>0$ 
such that its elements satisfy
\begin{equation}\label{eq:four170}
| a_{m,n} | \leq  c_L e^{-\eta_L\vert m - m \vert }\;  \qquad m, n  \geq 2  \;.
\end{equation}
\end{definition}
The following properties hold.
\begin{property}[{continuity of $\bA$}]\label{prop:bounded}
If ${\bf A}\in{\mathcal D}_e(\eta_L)$, then ${\bf A}$ defines a bounded operator
on {$\ell^2(\mathbb{N}_2)$}.
\end{property}
\begin{proof}
It is sufficient to extend the semi-infinite matrix $\mathbf{A}=(a_{\ell,k})_{\ell,k\in\mathbb{N}_2}$ to a bi-infinite matrix $\mathbf{\tilde{A}}=(\tilde{a}_{\ell,k})_{\ell,k\in\mathbb{Z}}$ such that it corresponds to the identity matrix for $\ell,k \in \mathbb{Z}\setminus \mathbb{N}_2$ and it is equal to $\mathbf{A}$ otherwise. Then proceed as in \cite{Jaffard:1990, Dahlke-Fornasier-Groechenig:2010}. 
\end{proof}

\begin{property}[{inverse of $\bA$}] \label{prop:inverse.matrix-estimate}
If $\mathbf{A} \in {\mathcal D}_e(\eta_L)$ and there exists a constant $c_L$
satisfying \eqref{eq:four170} such that 
\begin{equation}\label{restriction-cL}
c_L < \frac12({\rm e}^{\eta_L} -1) \min_\ell a_{\ell,\ell}\;,
\end{equation}
then ${\bf A}$ is invertible in  $\ell^2(\mathbb{N}_2)$ and 
${\bf A}^{-1}\in{\mathcal D}_e(\bar{\eta}_L)$ where 
$\bar{\eta}_L \in (0,\eta_L]$ 
is such that $\bar{z}={\rm e}^{-\bar{\eta}_L}$ is the unique zero in the 
interval $(0,1)$ of the polynomial
$$
z^2- \frac{{\rm e}^{2\eta_L}+2c_L+1}{{\rm e}^{\eta_L}(c_L+1)}z+1 \;.
$$
\end{property}
\begin{proof}
See \bnew{\cite[Property 2.3]{CNV:2011}}.
\end{proof}

\noindent For any integer $J \geq 0$, let $\mathbf{A}_J$ denote {the following
symmetric truncation of the matrix $\mathbf{A}$}
\begin{equation}\label{eq:trunc-matr}
(\mathbf{A}_J)_{\ell,k}=
\begin{cases}
a_{\ell,k} & \text{if } |\ell-k| \leq J \;, \\
0 & \text{elsewhere.}
\end{cases}
\end{equation}
Then, we have the following well-known results.
\begin{property}[truncation]\label{prop:matrix-estimate}
{The truncated matrix $\bA_J$ has a number of non-vanishing entries
bounded by $2 J$. Moreover,
under the assumption of Property \ref{prop:bounded},} there exists a constant 
$C_{\mathbf{A}} $ such that 
\[
\Vert \mathbf{A}-{\mathbf{A}}_J \Vert \leq
\psi_{\mathbf{A}}(J,\eta):=C_{\mathbf{A}} {\rm e}^{-\eta_L J}
\]
{for all $J\ge0$.}
Consequently, under the assumptions of Property \ref{prop:inverse.matrix-estimate},
one has
\begin{equation}\label{eq:trunc-invmatr-err}
\Vert \mathbf{A}^{-1}-(\mathbf{A}^{-1})_J \Vert \leq 
\psi_{\mathbf{A}^{-1}} (J,\bar{\eta}_L)
\end{equation}
where we {let $\bar\eta_L$ be defined in Property \ref{prop:inverse.matrix-estimate}.}
\end{property}

\begin{proof}
See \bnew{\cite[Property 2.4]{CNV:2011}}.
\end{proof}

%%%%%%%%%%%%%%%%%%%%%%%%%%%%%%%%%%%%%%%%%%%
\section{Towards an adaptive algorithm}\label{sec:plain-adapt-alg}
In order to design an adaptive algorithm with optimal convergence and complexity properties, we start 
by considering an {\sl ideal one}. 
\bnew{This} will serve as a reference to discuss the most relevant aspects which have to be taken into account in designing the final algorithm.
The ideal algorithm uses as
error estimator the ideal one, i.e., the norm of the residual in $H^{-1}(I)$.
\bnew{We} thus set, for any $v \in H^1_0(I)$,
\begin{equation}\label{eq:four.2.2}
\eta^2(v)=\Vert r(v) \Vert^2 = \sum_{k \in \mathbb{N}_2} |\hat{r}_k(v)|^2\;,
\end{equation}
so that (\ref{eq:four.2.1}) can be rephrased as
\begin{equation}\label{eq:four.2.3}
\frac1{\alpha^*} \eta(u_\Lambda) \leq
\Vert u - u_\Lambda \Vert \leq 
\frac1{\alpha_*} \eta(u_\Lambda) \; .
\end{equation}

Obviously, this estimator is hardly computable in practice. However,
by introducing 
\bnew{suitable polynomial approximations} of the coefficients and the
right-hand side, it is possible to consider a \bnew{\it feasible} version. In the sequel we will not pursue this possibility, but we refer to \cite{CNV:2011} for the details.
Given any subset $\Lambda \subseteq \mathbb{N}_2$, we also define the quantity
$$
\eta^2(v;\Lambda) = \Vert P_\Lambda r(v) \Vert^2 
= \sum_{k \in \Lambda} |\hat{r}_k(v)|^2\;,
$$
so that $\eta(v)=\eta(v;\mathbb{N}_2)$.

We now introduce the following procedures, which will enter the definition of all our adaptive algorithms.

\begin{itemize}
\item $u_\Lambda := {\bf GAL}(\Lambda)$ \\
Given a finite subset $\Lambda \subset\mathbb{N}_2$, the output
$u_\Lambda \in V_\Lambda$ is the solution of the Galerkin problem (\ref{eq:four.2}) relative to $\Lambda$.

\item $r := {\bf RES}(v_\Lambda)$ \\
Given a function $v_\Lambda \in V_\Lambda$ for some finite index set $\Lambda$, 
the output $r$ is the residual $r(v_\Lambda)=f-Lv_\Lambda$.

\item $\Lambda^* := \text{\bf D\"ORFLER}(r, \theta)$\\
Given $\theta \in (0,1)$ and an element $r \in H^{-1}(I)$, 
the ouput $\Lambda^* \subset \mathbb{N}_2$ is a finite set of minimal cardinality 
such that the inequality
\begin{equation}\label{eq:four.2.5.5}
\Vert P_{\Lambda^*} r \Vert \geq \theta  \Vert r \Vert 
\end{equation}
is satisfied.
\end{itemize}
Note that the latter inequality is equivalent to
\begin{equation}\label{eq:four.2.5.5bis}
\Vert r-P_{\Lambda^*} r \Vert \leq \sqrt{1-\theta^2}  \Vert r \Vert \;.
\end{equation}
If $r=r(u_\Lambda)$ is the residual of a Galerkin solution $u_\Lambda \in V_\Lambda$, 
then by (\ref{eq:four.2.1ter}) we can trivially assume 
that $\Lambda^*$ is contained in $\Lambda^c := \mathbb{N}_2 \setminus \Lambda$. 
For such a residual, inequality (\ref{eq:four.2.5.5}) can then be stated as
\begin{equation}\label{eq:four.2.5.5ter}
\eta(u_\Lambda;\Lambda^*) \geq \theta \eta(u_\Lambda) \;
\end{equation}
\bnew{{\bf D\"orfler marking} (or bulk chasing)}. Writing $\hat{r}_k = \hat{r}_k(u_{\Lambda})$,
the condition {\eqref{eq:four.2.5.5ter}} can be equivalently stated as
\begin{equation}\label{eq:four.2.4.bis}
\sum_{k \in \Lambda^*}  |\hat{r}_k|^2 
\geq \theta^2  \sum_{k \not \in \Lambda}  |\hat{r}_k|^2 \;. 
\end{equation}
Thus, the output set $\Lambda^*$ of minimal cardinality can be immediately determined 
if the coefficients $\hat{r}_k$ are rearranged in non-increasing order of modulus. 
The cardinality of $\Lambda^*$ depends on the rate of decay of the
rearranged coefficients, i.e., on  \bnew{{\sl the sparsity} of 
the} representation of the residual in the chosen basis.

\medskip
Given two parameters $\theta \in (0,1)$ and $tol \in [0,1)$, 
we are ready to define  our ideal adaptive algorithm. 

\bigskip
{\bf Algorithm ADLEG}($\theta, \, tol$)
\begin{itemize}
\item[\ ] Set $r_0:=f$, $\Lambda_0:=\emptyset$, $n=-1$
\item[\ ] do
	\begin{itemize}
	\item[\ ] $n \leftarrow n+1$
	\item[\ ] $\partial\Lambda_{n}:= \text{\bf D\"ORFLER}(r_n, \theta)$
	\item[\ ] $\Lambda_{n+1}:=\Lambda_{n} \cup \partial\Lambda_{n}$
	\item[\ ] $u_{n+1}:= {\bf GAL}(\Lambda_{n+1})$
	\item[\ ] $r_{n+1}:= {\bf RES}(u_{n+1})$
	\end{itemize}
\item[\ ]  while $\Vert r_{n+1} \Vert > tol $
\end{itemize}

This algorithm is the non-periodic counterpart of the ideal algorithm {\bf ADFOUR} considered in \cite{CNV:2011}. The same proof  given therein yields the following result,
which states the convergence of the algorithm with a guaranteed error reduction rate.

\begin{theorem}[{contraction property of {\bf ADLEG}}]\label{teo:four1}
Let us set 
\begin{equation}\label{eq:def_rhotheta}
\rho=\rho(\theta)= \sqrt{1 - \frac{\alpha_*}{\alpha^*}\theta^2} \in (0,1) \;.
\end{equation}
Let $\{\Lambda_n, \, u_n \}_{n\geq 0}$ be the sequence generated by
the adaptive algorithm {\bf ADLEG}.
Then, the following bound holds for any $n$:
$$
  \tvert u-u_{n+1} \tvert \leq \rho \tvert u-u_n \tvert \;.
$$ 
Thus, for any $tol>0$ the algorithm terminates in a finite number of iterations, whereas for $tol=0$
the sequence $u_n$ converges to $u$ in $H^1(I)$ as $n \to \infty$. \endproof
\end{theorem}

\noindent At this point, some comments are in order.
\begin{itemize}
\item The predicted error reduction rate $\rho=\rho(\theta)$, being bounded 
from below by the quantity $\sqrt{1 - \frac{\alpha_*}{\alpha^*}}$, looks overly pessimistic, particularly in the case of smooth (analytic) solutions. Indeed, in this case a spectral (Legendre) Galerkin method is expected to exhibit very fast (exponential)
error decay. For this reason, we will introduce in Section \ref{sec:moddorfler} a variant of the D\"orfler procedure 
which -- through a suitable enrichment of the set of new degrees of freedom detected by 
the usual D\"orfler -- will guarantee an arbitrarily large error reduction per iteration.
\item The complexity analysis of the algorithm requires to relate the current error $\varepsilon_n:=\tvert u - u_n\tvert$ to the cardinality of the set $\Lambda_n$ of the activated degrees of freedom, having as a benchmark the best approximation (i.e. the one achieved with the minimal number of degrees of freedom) of the exact solution $u$ up to an error given exactly by $\varepsilon_n$. This requires to investigate the {\sl sparsity class} of the solution $u$, a task that will be accomplished in Section \ref{sec:nl}; 
we will confine ourselves to the case of infinite differentiable functions (including analytic ones) for which a natural framework is provided by Gevrey spaces. (The analysis of the case of finite smoothness can be carried out by extending the arguments presented in \cite{CNV:2011}.) 
\item The cardinality of the set $\Lambda_n$ of active degrees of freedom selected by the D\"orfler procedure may be estimated in terms of the sparsity class of the residual, rather than the one of the solution. If the residual is less sparse than the solution, we run into a potential situation of non-optimality.
This is precisely what happens for the Gevrey-type sparsity classes, as pointed out in Section \ref{sec:spars-res}. For this reason, we will incorporate in our algorithm a coarsening step, introduced in Section \ref{sec:coarsening}, to bring the cardinality of the active degrees of freedom at the end of each iteration to be comparable with the optimal one dictated by the sparsity class of the solution.
\end{itemize}

%%%%%%%%%%%%%%%%%%%%%%%%%%%%%%%%%%%%%%%%%%%
\section{Nonlinear approximation in Gevrey spaces}\label{sec:nl}
At first, we consider Gevrey spaces of linear type and then, through the concept of nonlinear approximation,  we will introduce 
sparsity classes of functions related to Gevrey spaces.
\subsection{Gevrey classes and their properties}
We recall the following definition of classical Gevrey space. 
\begin{definition}[definition of ${G}^t(\bar{I})$]\label{def:Gt}
For any $t\in (0,1]$, we denote by $G^t(\bar{I})$ the space of $C^\infty$ functions $v$ 
in a neighborhood of $\bar{I}$ for which there exist a constant $L\geq 0$
such that for any $n\geq 0$
\begin{equation}
\| D^n v\|_{L^2(I)} \leq L^{n+1} (n !)^{1/t}\ .
\end{equation}
\end{definition}
The choice $t=1$ yields the usual class $\mathscr{A}(\bar{I})=G^1(\bar{I})$ 
of analytic functions in a neighborhood of $\bar{I}$.

Another family of spaces of Gevrey type has been introduced in \cite{BG:72} by relaxing the requirement on the growth of the derivatives near the boundary. In our simple one-dimensional setting, its definition is a follows.
\begin{definition}[definition of $\mathscr{A}^t(\bar{I})$]\label{def:At}
Let $\mathcal{L}$ denote the Legendre operator in $I$
$$
\mathcal{L}v =-D((1-x^2)D)v \;;
$$ 
and, for any $k\geq 0$, let us set 
\begin{equation*}
R_k v= 
\begin{cases}
\mathcal{L}^p v &\textit{if~}k=2p,\\
(1-x^2)D\mathcal{L}^p v & \textit{if~}k=2p+1\ .
\end{cases}
\end{equation*}
Then for any $t\in (0,1]$, we denote by $\mathcal{A}^t(\bar{I})$ the space of $C^\infty$ functions $v$ 
in a neighborhood of $\bar{I}$ for which there exist a constant $L\geq 0$
such that for any $k\geq 0$
\begin{equation}
\| R_k v\|_{L^2(I)} \leq L^{k+1} (k !)^{1/t}\ .
\end{equation}
\end{definition}

\noindent The relation between the two families is given by the inclusion 
\begin{equation}\label{inclusion}
{G}^t(\bar{I})\subseteq\mathcal{A}^t(\bar{I})\subseteq {G}^{\frac{t}{2-t}}(\bar{I})\ ,
\end{equation}
where the first inclusion is an immediate consequence of the fact that $1-x^2$ is bounded with its derivatives in $\bar{I}$, whereas the second one is proved in \cite{BG:72}.

Let $\mathbb{P}_n=\{\phi_0,\cdots,\phi_n\}$ be the space of polynomials of degree 
$\leq n$ in $I$ and let $d_2(v,\mathbb{P}_n)=\inf_{p\in\mathbb{P}_n}\| v-p\|_{L^2(I)}$
be the distance of a function $v\in L^2(I)$ \bnew{to} $\mathbb{P}_n$.
\begin{proposition}\label{prop:gevrey}
Let $v\in L^2(I)$. The following conditions are equivalent:
\begin{enumerate}
\item[(i)] $v\in \mathscr{A}^t(I)$.
\item[(ii)] There exist $L_1\geq 0$ and $\eta_1 >0$ such that
\begin{equation}\label{eq:gevrey-error}
d_2(v,\mathbb{P}_n)\leq L_1 e^{-\eta_1 n^t}\quad \forall n\geq 0.
\end{equation} 
\item[(iii)] There exist $L_2\geq 0$ and $\eta_2 >0$ such that the Legendre coefficients
of $v=\sum_{k=0}^\infty \hat{v}_k \phi_k$ satisfy
\begin{equation}\label{eq:gevrey-coeff}
\vert \hat{v}_k \vert \leq L_2 e^{-\eta_2 k^t}\quad \forall k\geq 0.
\end{equation}
\end{enumerate}
\end{proposition}
\begin{proof}
The equivalence $(i)\leftrightarrow (ii)$ is proved in 
\cite[Thm. 7.2]{BG:72}. In order to prove the equivalence 
$(ii)\leftrightarrow (iii)$, we observe that 
\begin{equation}
d_2(v,\mathbb{P}_n)=\| v - \sum_{k=0}^n \hat{v}_k \phi_k\|_{L^2(I)}
=\left( \sum_{k>n} \vert \hat{v}_k\vert^2\right)^{1/2}\ .
\end{equation}
Thus, if $(ii)$ holds, we have 
\begin{equation}
\bnew{\vert \hat{v}_{n+1}\vert} \leq d_2(v,\mathbb{P}_n)\leq L_1 e^{-\eta_1 n^t}\ .
\end{equation}
Now, recalling that $(a+b)^t\leq a^t +b^t$ for $a,b\geq 0$ and $t\in (0,1]$, we have
\begin{equation}
e^{-\eta_1 n^t} = \bnew{e^{\eta_1}} e^{-\eta_1 (n^t +1)}=\bnew{e^{\eta_1}} e^{-\eta_1 (n^t +1^t)}\leq 
\bnew{e^{\eta_1}} e^{-\eta_1 (n+1)^t}
\end{equation}
which yields $(iii)$ with $L_2=e^{\eta_1} L_1$ and $\eta_2=\eta_1$. Conversely, let us assume that 
$(iii)$ holds. Then, using again \eqref{eq:gevrey-error} and setting for simplicity $\eta=\eta_2$, we have
\begin{equation}
d_2(v,\mathbb{P}_n)^2 \leq L_2^2 \sum_{k>n} e^{-2\eta k^t} \lesssim 
L_2^2 \int_{n}^{+\infty} e^{-2\eta y^t} dy\ .
\end{equation}
Now, setting $z=y^t$ and $s=1/t \geq 1$, we have
\begin{equation}
\begin{split}
S_n&:=\int_{n}^{+\infty} e^{-2 \eta y^t} dy = s \int_{n^t}^{+\infty} e^{-2\eta z} z^{s-1} dz\\
&= s \, e^{-2\eta n^t} \int_{n^t}^{+\infty} e^{-2\eta (z-n^t)} z^{s-1} dz
~= s \, e^{-2\eta n^t} \int_{0}^\infty e^{-2\eta w} (z+n^t)^{s-1} dw \ .
\end{split}
\end{equation}
The last integral can be bounded by a polynomial of degree $\lceil  s-1 \rceil$. It follows 
that for any $\eta_1 < \eta=\eta_2$ we can find $C_{\eta_1,s}>0$ such that 
$$ S_n \leq C_{\eta_1,s} e^{-2\eta_1 n^t},$$
thus condition $(ii)$ is satisfied.
\end{proof}
\begin{remark}
\rm As a consequence of the previous proposition and the inclusion 
$G^t (\bar{I}) \subseteq \mathscr{A}^t(\bar{I})$, any Gevrey function $v\in G^t(\bar{I})$
admits Legendre coefficients $\hat{v}_k$ which decay according to the law 
\eqref{eq:gevrey-coeff}.
\end{remark}
For analytic functions, the results in $(ii)$ and $(iii)$ of Proposition \ref{prop:gevrey}
can be made more precise as follows.
\begin{proposition}
Let $v\in L^2(I)$ be analytic in the closed ellipse $\mathcal{E}_r\supset [-1, 1]$  
in the complex plane $\mathbb{C}$ with foci in $\pm 1$ and semiaxes' sum equal to $r>1$. 
Then, one has 
\begin{equation}\label{eq:anal-coeff}
\vert \hat{v}_k\vert \leq C(r)\sqrt{2(2k+1)} e^{-{\eta}k}\quad \forall k\geq 0\ ,
\end{equation}
and
\begin{equation}\label{eq:d2}
d_2(v,\mathbb{P}_n)\leq \tilde{C}(r) (n+1)^{1/2} e^{-\eta n}\quad \forall n\geq 0
\end{equation}
for constants $C(r)$ and $\tilde{C}(r)$ only depending on $r$ and ${\eta}=\log r$. 
\end{proposition}
\begin{proof}
The estimate \eqref{eq:anal-coeff} is a consequence of the bound 
$$ 
\vert \hat{v}_k^L \vert \leq C(r) (2k+1) e^{-\eta k} \quad k\geq 0 
$$ 
given in \cite[Theorem 12.4.7]{Davis:1963}, where 
$\hat{v}_k^L = (k+1/2)\int_{-1}^1 v(x) L_k(x) dx $ and satisfies 
$\hat{v}_k^L=\sqrt{k+1/2}\, \hat{v}_k$. On the other hand, 
\begin{eqnarray}
d_2(v,\mathbb{P}_n)^2 &=& \sum_{k>n} \vert \hat{v}_k\vert^2
~\leq 2 C^2(r) \sum_{k>n} (2k+1) e^{-2\eta k}\nonumber\\
&\leq& 2 C^2(r) \int_n^{+\infty} (2y+3) e^{-2\eta y} dy
~ \leq 6 \frac{C^2(r)}{\eta} (n+1) e^{-2\eta n}\nonumber, 
\end{eqnarray}
\bnew{and \eqref{eq:d2} follows.}
\end{proof}

\noindent There follows that conditions $(ii)$ and $(iii)$ are fulfilled with $t=1$ and any $\eta_1,\eta_2< \log r$.
%%%%%%%%%%%%%%%%%%%%%%%%%%%%%%%%%%%%%%%%%%%
\subsection{Nonlinear approximation and sparsity classes}
From now on, we represent any $v\in V$ according to the BS basis 
\bnew{$\{\eta_k\}_{k=2}^\infty$} as 
$v=\sum_{k=2}^\infty \hat{v}_k \eta_k(x)$.
We recall that given any nonempty finite index set $\Lambda \subset \mathbb{N}_{2}$ and the corresponding subspace
$V_\Lambda \subset V$ of dimension $|\Lambda|=\text{card}\, \Lambda$, 
the best approximation of $v$ in $V_{\Lambda}$ is the orthogonal projection
of $v$ upon $V_{\Lambda}$, i.e. the function $P_{\Lambda} v = \sum_{k \in \Lambda} \hat v_k \eta_k$,
which satisfies
$$
\Vert v - P_{\Lambda} v \Vert= 
\left(\sum_{k \not \in \Lambda} |\hat{v}_k|^2\right)^{1/2} 
$$
(we set $P_\Lambda v = 0$ if $\Lambda=\emptyset$).
For any integer $N \geq 1$, we minimize this error over all possible choices of $\Lambda$
with cardinality $N$, {thereby} leading to the {\sl best $N$-term approximation error}
$$
E_N(v)= \inf_{\Lambda \subset\mathbb{N}_2 , 
\ |\Lambda|=N} \Vert v - P_{\Lambda} v \Vert \;.
$$
A way to construct a {\sl best $N$-term approximation} $v_N$ of $v$ consists of
rearranging the coefficients of $v$ in decreasing order of modulus
$$
\vert \hat{v}_{k_1}\vert \geq \ldots \geq \vert \hat{v}_{k_n} \vert 
\geq \vert \hat{v}_{k_{n+1}} \vert \geq \dots 
$$
and setting $v_N=P_{\Lambda_N}v$ with $\Lambda_N = \{ k_n \ : \ 1 \leq
n \leq N \}$. 
Let us denote from now on $v_n^*=\hat{v}_{k_n}$
the rearranged BS coefficients of $v$. Then,
$$
E_N(v)= \Vert v - \bnew{P_{\Lambda_N} v} \Vert =  \left(\sum_{n>N} |v_n^*|^2 \right)^{1/2} \;.
$$

We will be  interested in classifying functions according to the decay law of their best $N$-term approximations, as $N\to\infty$. To this end, given a strictly decreasing function $\phi : \mathbb{N} \to \mathbb{R}_+$ such that
$\phi(0)=\phi_0$ for some $\phi_0 >0$
and $\phi(N) \to 0$ when $N \to \infty$, we introduce the corresponding
 {\sl sparsity class} ${\mathcal A}_\phi$ by setting
\begin{equation}\label{eq:nl.gen.001}
{\mathcal A}_\phi = {\Big\{ v \in V \ : \ \Vert v \Vert_{{\mathcal A}_\phi}:= \sup_{N \geq 0} 
\, \frac{E_N(v)}{\phi(N)} < +\infty \Big\} \; .}
\end{equation}
The quantity $\Vert v \Vert_{{\mathcal A}_\phi}$ (which need not be a (quasi-)norm, since
${\mathcal A}_\phi$ need not be a linear space) dictates
the minimal number $N_\varepsilon$ of basis functions
needed to approximate $v$ with accuracy $\varepsilon$. {In fact}, from the relations
$$
E_{N_\varepsilon}(v) \leq \varepsilon  < E_{N_\varepsilon-1}(v) 
\leq \phi(N_\varepsilon-1) \Vert v \Vert_{{\mathcal A}_\phi} \;,
$$
{and the monotonicity of $\phi$, we obtain}
\begin{equation}\label{eq:nl.gen.2o}
N_\varepsilon \leq \phi^{-1}\left(\frac{\varepsilon}{\Vert v \Vert_{{\mathcal A}_\phi}}\right) +1 \;.
\end{equation}

Hereafter, we focus on sparsity classes identified by exponentially
decaying $\phi$ of the type $\phi(N)=e^{-\eta N^t}$ for real
\bnew{numbers} $\eta>0$ and $0<t\leq 1$. The following argument
clarifies their relationship with certain families of Gevrey
spaces. Let us introduce the following spaces of Gevrey type: given
real \bnew{numbers} $\eta>0$ and $0<t\leq 1$, we set 
\begin{equation}
G^{\eta,t}(I)=\{v\in V:\quad \| v\|_{G,\eta,t}:= \sum_{k=2}^\infty e^{2\eta k^t} 
\vert \hat{v}_k\vert^2  < + \infty\}\ .
\end{equation}
We immediately observe that $v \in G^{\eta,t}(I)$ implies 
$$ \vert \hat{v}_k\vert^2 \leq e^{- 2\eta k^t} \| v\|_{G,\eta,t}\quad \forall k\geq 2 .$$
Using \eqref{aux:coeff}, we deduce from Proposition \ref{prop:gevrey} that the first derivative of a function $v\in G^{\eta,t}(I)$ belongs to the space $\mathscr{A}^t(\bar{I})$ defined in 
Definition \ref{def:At}; in particular, by using \eqref{inclusion}, $v^{\prime}$ is a Gevrey function in $G^{\frac{t}{2-t}}(\bar{I})$.

Functions in $G^{\eta,t}({I})$ can be approximated by the linear orthogonal projection
on $\mathbb{P}_M^0(I):=\mathbb{P}_M(I)\cap H^1_0(I)$
$$
P_M v = \sum_{k=2}^M \hat{v}_k \eta_k \;,
$$
for which we have
\begin{eqnarray*}
\Vert v -P_M v \Vert^2 &=& \sum_{k>M} |\hat{v}_k|^2 
= \sum_{k > M} \e^{-2\eta k^t} \e^{2\eta k^t} |\hat{v}_k|^2  \\
&\leq & \e^{-2\eta M^t} \sum_{k > M} \e^{2\eta k^t} |\hat{v}_k|^2 \leq 
\e^{-2\eta M^t} \Vert v \Vert_{G,\eta,t}^2 \;.
\end{eqnarray*}
Setting $N:=\text{dim}(\mathbb{P}_M^0)=M-1$, we have 
\begin{equation}\label{eq:nlg.3}
E_N(v) \leq \Vert v -P_{N+1} v \Vert \lsim 
e^{- \eta (N+1)^{t} } \Vert v \Vert_{G,\eta,t}
\leq e^{- \eta N^{t}} \Vert v \Vert_{G,\eta,t} \;.
\end{equation}
Hence, we are led to introduce the function 
\begin{equation}\label{eq:nlg.300}
\phi(N)=e^{- \eta  N^{t} } \quad \qquad (N \geq 1) \;,
\end{equation} 
whose inverse is given by
\begin{equation}\label{eq:nlg.301}
\phi^{-1}(\lambda) = \frac{1}{\eta^{1/t}}\left( \log \frac1\lambda \right)^{1/t}
\quad \qquad (\lambda \leq 1) \;,
\end{equation}
and to consider the corresponding class ${\mathcal A}_\phi$ defined in (\ref{eq:nl.gen.001}), 
which contains $G^{\eta,t}(I)$.
\begin{definition}[{exponential class of functions}]\label{def:AGev} 
We denote by ${\mathcal A}^{\eta,t}_G$ the \bnew{set} defined as
$$
{\mathcal A}^{\eta,t}_G { := \Big\{ v \in V \ : 
\ \Vert v \Vert_{{\mathcal A}^{\eta,t}_G}:= 
\sup_{N \geq 0} \, E_N(v) \, e^{\eta N^{t} } < +\infty \Big\} \;.}
$$
\end{definition}

%%%%%%%%%%%%%%%%%%%%%%%%%%%%%%%%%%%%%%%%%%%%%%%%%%%
%%%%%%%%%%%%%%%%%%%%%%%%%%%%%%%%%%%%%%%%%%%%%%%%%%%
%%%%%%%%%%%%%%%%%%%%%%%%%%%%%%%%%%%%%%%%%%%%%%%%%%%
The class ${\mathcal A}^{\eta,t}_G$ can be equivalently characterized in terms of 
behavior of rearranged sequences of BS coefficients. 
\begin{definition}[{exponential class of sequences}]\label{def:elpicGev}
{Let $\ell_G^{\eta,t}(\mathbb{N}_2)$ be the} subset of sequences 
${\bv} \in \ell^{2}(\mathbb{N}_2)$ so that \looseness=-1
$$
\Vert {\bv} \Vert_{\ell_G^{\eta,t}(\mathbb{N}_2)} := \sup_{n \geq 1} 
\bnew{\Big( n^{(1-t)/2} {\rm exp}\left(\eta n^{t} \right)
%\e^{\eta \omega_d^{t/d} n^{t/d}} 
|v_n^*| \Big)} < +\infty \;,
$$
where {${\bv}^*=(v_n^*)_{n=1}^\infty$} is the non-increasing rearrangement of ${\bv}$.
\end{definition}

%%%%%%%%%%%

The relationship between  ${\mathcal A}^{\eta,t}_G$ and $\ell_G^{\eta,t}(\mathbb{N}_2)$ is stated in the following Proposition, whose proof is a straightforward adaptation 
of the one given in \cite[Proposition 4.2]{CNV:2011}.
\begin{proposition}[{equivalence of exponential classes}]\label{prop:nlg1}
Given a function $v \in V$ and the sequence ${\bm v}=(\hat{v}_k)_{k \in \mathbb{N}_2}$ of its BS coefficients,
 one has  $v \in {\mathcal A}^{\eta,t}_G$ if and only if
${\bv} \in \ell_G^{\eta,t}(\mathbb{N}_2)$, with
$$
\Vert v \Vert_{{\mathcal A}^{\eta,t}_G} \lsim 
\Vert {\bv} \Vert_{\ell_G^{\eta,t}(\mathbb{N}_2)}
\lsim \Vert v \Vert_{{\mathcal A}^{\eta,t}_G}\,.
$$
\end{proposition}

As shown in \cite{CNV:2011} the set $\ell_G^{\eta,t}(\mathbb{N}_2)$ is not
a vector space, since it may happen that ${\bu},\,{\bv}$ belong to this set, whereas
${\bu}+{\bv}$ does not. On the other hand, we have the following property 
(see \cite[Lemma 4.1]{CNV:2011}).

\begin{property}[{quasi-triangle inequality}]\label{L:nlg2}
If ${\bu}_i\in {\ell_G^{\eta_i,t}(\mathbb{N}_2)}$ for $i=1,2$, then 
${\bu}_1+{\bu}_2 \in {\ell_G^{\eta,t}(\mathbb{N}_2)}$ with
\[
{\|{\bu}_1+{\bu}_2\|_{\ell_{G}^{\eta,t}(\mathbb{N}_2)} \le
\|{\bu}_1\|_{\ell_{G}^{\eta_1,t}(\mathbb{N}_2)} + \|{\bu}_2\|_{\ell_{G}^{\eta_2,t}(\mathbb{N}_2)},}
\qquad
\eta^{-\frac{1}{t}} = \eta_1^{-\frac{1}{t}} +
\eta_2^{-\frac{1}{t}}.
\]
\end{property}
%

%%%%%%%%%%%%%%%%%%%%%%%%%%%%%%%%%%%%%%%%%%%%%%%%%%%
%%%%%%%%%%%%%%%%%%%%%%%%%%%%%%%%%%%%%%%%%%%%%%%%%%%%%%%%%%%%%%%%%%%%%%%%%%%%%%%%%%%%%%
\section{Sparsity classes of the residual \bnew{and the solution}}\label{sec:spars-res}
%%%%%%%%%%%%%%%%%%%%%%%%%%%%%%%%%%%%%%%%%%%%%%%%%%%%%%%%%%%%%%%%%%%%%%%%%%%%%%%%%%%%%%
This section is devoted to the study of the sparsity class of the residual
$r=r(u_\Lambda)$ produced by any Galerkin solution $u_\Lambda\in V_\Lambda$, and its connection with the sparsity class of the exact solution $u$. Indeed, the step
$$
\partial\Lambda := \text{\bf D\"ORFLER}(r, \theta)
$$
selects a set {$\partial\Lambda$}
of minimal cardinality in $\Lambda^c$ for which $\Vert r-P_{\partial\Lambda}r \Vert
\leq \sqrt{1-\theta^2} \Vert r \Vert$. Thus, 
if $r$ belongs to a certain sparsity class ${\mathcal A}_{\bar{\phi}}$, identified by a function $\bar{\phi}$ according to \eqref{eq:nl.gen.001}, 
then (\ref{eq:nl.gen.2o}) yields
\begin{equation}\label{eq:boundDorfler} 
|\partial\Lambda| \leq {\bar{\phi}}^{-1}\left( \, 
\frac{\sqrt{1-\theta^2}~\Vert r \Vert}{\ \Vert r \Vert_{{\mathcal A}_{\bar{\phi}}}}\right) +1 \;.
\end{equation}
\bnew{Specifically}, if $r \in {\mathcal A}^{\bar{\eta},\bar{t}}_G$ for some $\bar{\eta}>0$ and $\bar{t}>0$, we have by (\ref{eq:nlg.301})
\begin{equation}\label{aux:residual}
|\partial\Lambda| \leq \frac{1}{\bnew{\bar\eta}^{1/\bar{t}}} \left(
 \log \frac{\Vert r \Vert_{{\mathcal A}^{\bar{\eta},\bar{t}}_G}}{\sqrt{1-\theta^2}~\Vert r \Vert} \right)^{1/\bar{t}} + 1 \;.
\end{equation}

We begin by investigating the sparsity class of the image $Lv$, when
the function $v$ belongs to the sparsity class
$\mathcal{A}^{\eta,t}_G$. 
\bnew{It turns out that} the sparsity classes of $v$ and $Lv$ are equivalent, in view of Proposition \ref{prop:nlg1}, to 
the sparsity classes of the related 
vectors $\mathbf{v}$ and $\mathbf{A} \mathbf{v}$, where $\mathbf{A}$ is the stiffness matrix (\ref{eq:four100}).
\begin{proposition}[{continuity of $L$ in $\mathcal{A}^{\eta,t}_G$}]\label{propos:spars-res}
Let the differential operator $L$ be such that the corresponding stiffness matrix satisfies 
$\mathbf{A} \in {\mathcal D}_e(\eta_L)$ for some constant $\eta_L>0$.
Assume that $v \in {\mathcal A}^{\eta,t}_G$ for some $\eta>0$ and $t \in (0,1]$. 
Let one of the two following set of conditions be satisfied.
\begin{enumerate}[\rm (a)]
\item
{If the} matrix $\mathbf{A}$ is banded with $2p+1$ non-zero
diagonals, {let us set}
$$
\bar{\eta}= \frac{\eta}{(2p+1)^t} \;, \qquad \bar{t}= t \;.
$$
\item
{If the matrix $\mathbf{A}$ is dense}, but the coefficients $\eta_L$ and $\eta$ satisfy
the inequality $\eta< \eta_L $, {let us set} 
$$
\bar{\eta}= \zeta(t)\eta \;, \qquad \bar{t}= \frac{t}{1+t} \;,
$$
where we define 
\begin{equation}\label{aux-funct}
\zeta(t) := \Big( \frac{1+t}{2} \Big)^{\frac{t}{1+t}}
\qquad\forall\; 0<t\le 1.
\end{equation}
\end{enumerate}
Then, one has $Lv \in {\mathcal A}^{\bar{\eta},\bar{t}}_G$, with
\begin{equation}\label{eq:spars11bis}
\Vert Lv \Vert_{{\mathcal A}_G^{\bar{\eta},\bar{t}}} \lsim 
\Vert v \Vert_{{\mathcal A}_G^{\eta,t}} \;.
\end{equation}
\end{proposition}
\proof
The proof is an adaptation of a similar result in the periodic case given in \cite[Proposition 5.3]{CNV:2011}; we report here the details for completeness. Let $\mathbf{A}_J$ be the truncation of the stiffness matrix defined in 
(\ref{eq:trunc-matr}); thus, by Property \ref{prop:matrix-estimate} we have
$
\Vert \mathbf{A}-\mathbf{A}_J \Vert \leq
C_{\mathbf{A}} {\rm e}^{-\eta_L J}, 
$
$J\geq 0$.
On the other hand, for any $j \geq 1$, let $\mathbf{v}_j=\mathbf{P}_j(\mathbf{v})$ be a best $j$-term approximation of $\mathbf{v}$ (with $\mathbf{v}_{0}=0$), 
which therefore satisfies $\Vert \mathbf{v}-\mathbf{v}_j \Vert \leq {\rm e}^{-\eta  j^t} 
\Vert {\bv} \Vert_{\ell_G^{\eta,t}(\mathbb{N}_2)}$. Note that the difference $\bv_j -\bv_{j-1}$ 
consists of a single component and satisfies as well
$$
\Vert \bv_j-\bv_{j-1} \Vert \lsim {\rm e}^{-\eta  j^t} 
\Vert {\bv} \Vert_{\ell_G^{\eta,t}(\mathbb{N}_2)} \;. 
$$
Finally, let us introduce the function $\chi \, : \, \mathbb{N} \to \mathbb{N}$ defined 
as $\chi(j)=\lceil j^t \rceil$, the smallest integer larger than or equal to $j^t$.

For any $J \geq 1$, let $\bw_J$ be the approximation of $\mathbf{A}\bv$ defined as
$$
\bw_J = \sum_{j=1}^J \mathbf{A}_{\chi(J-j)}(\bv_j-\bv_{j-1}) \;.
$$
Writing $\bv=\bv-\bv_J + \sum_{j=1}^J (\bv_j-\bv_{j-1})$, we obtain 
$$
\mathbf{A}\bv-\bw_J = \mathbf{A}(\bv-\bv_J)
+ \sum_{j=1}^J (\mathbf{A}-\mathbf{A}_{\chi(J-j)})(\bv_j-\bv_{j-1})\;.
$$ 
{We now assume to be in Case (b). Since
  $\mathbf{A}:\ell^2(\mathbb{N}_2)\to\ell^2(\mathbb{N}_2)$ is continuous}, the last equation yields
\begin{equation}\label{eq:spars11}
\Vert \mathbf{A}\bv-\bw_J \Vert \lsim \left( {\rm e}^{-\eta  J^t} 
+ \sum_{j=1}^J {\rm e}^{-(\eta_L \lceil (J-j)^t \rceil +\eta  j^t )}
\right) \Vert {\bv} \Vert_{\ell_G^{\eta,t}(\mathbb{N}_2)} \;.
\end{equation}
The exponents of the addends can be bounded from below as follows
{because $t\le 1$}
\begin{eqnarray*}
\eta_L \lceil (J-j)^t \rceil +\eta  j^t &=&
\eta_L \lceil (J-j)^t \rceil - \eta  (J-j)^t + \eta ( (J-j)^t + j^t) \\
&\geq& \eta_L (J-j)^t  - \eta (J-j)^t + \eta ((J-j) + j)^t \\
&=& \beta (J-j)^t + \eta J^t \;, 
\end{eqnarray*}
with $\beta= \eta_L -\eta  >0$ by assumption. Then,
(\ref{eq:spars11}) yields
\begin{equation}\label{eq:spars12}
\Vert \mathbf{A}\bv-\bw_J \Vert \lsim \left( 1 + \sum_{j=0}^{J-1} 
 {\rm e}^{-\beta j^t} \right) 
{\rm e}^{-\eta J^t} \Vert {\bv} \Vert_{\ell_G^{\eta,t}(\mathbb{N}_2)}\lsim
{\rm e}^{-\eta J^t} \Vert {\bv} \Vert_{\ell_G^{\eta,t}(\mathbb{N}_2)} \;.
\end{equation}

On the other hand, by construction $\bw_J$ belongs to a finite
dimensional space $\mathbf{V}_{\Lambda_J}$, where
{
\begin{equation}\label{eq:spars13}
|\Lambda_J| \leq 2 \sum_{j=1}^J \chi(J-j) = 
2 \sum_{j=0}^{J-1} \lceil j^t \rceil 
\sim \frac{2}{1+t} J^{1+t} \qquad \text{as } J \to \infty \;. 
\end{equation}
This implies
$$
\Vert \mathbf{A}\bv-\bw_J \Vert \lsim {\rm e}^{-\bar{\eta}  |\Lambda_J|^{\bar t}}
\Vert {\bv} \Vert_{\ell_G^{\eta,t}(\mathbb{N}_2)} \;,
$$
with $\bar t = \frac{t}{1+t}$ and 
$\bar\eta = \left(\frac{1+t}{2}\right)^{\bar t}\eta=\zeta(t)\eta$
as asserted.} 

We last consider Case (a). \bnew{Since $\mathbf{A}_{\chi(J-j)}=\mathbf{A}$ if
$\chi(J-j) \geq p$, for $\chi(J-j)\le p$, whence
$j \geq J-p^{1/t}$}, the summation in (\ref{eq:spars11}) can be limited to those $j$
satisfying $j_p \leq j \leq J$, where $j_p= \lceil J- p^{1/t} \rceil$. Therefore 
$$
\Vert \mathbf{A}\bv-\bw_J \Vert \lsim \left( {\rm e}^{-\eta J^t} 
+ \sum_{j=j_p}^J \bnew{{\rm e}^{- \eta  j^t }}
\right)\Vert {\bv} \Vert_{\ell_G^{\eta,t}(\mathbb{N}_2)} \;.
$$
Now, $J-j \leq p^{1/t}$ if $j_p\leq j \leq J$ and $j^t \geq j_p^t \geq (J-p^{1/t})^t \geq J^t - p$, whence
$$
\Vert \mathbf{A}\bv-\bw_J \Vert \lsim \left(1+ {\rm e}^{\eta p}
\right) {\rm e}^{-\eta  J^t} \Vert {\bv} \Vert_{\ell_G^{\eta,t}(\mathbb{N}_2)} \;.
$$
We conclude by observing that $|\Lambda_J|\leq (2p+1)J$, since any matrix $\mathbf{A}_J$ has 
at most $2p+1$ diagonals.  \endproof

Before studying the sparsity class of the residual, it is convenient to rewrite the Galerkin problem (\ref{eq:four.2}) in an equivalent (infinite-dimensional) way. To this end, let  $\mathbf{u}_\Lambda \in \mathbb{R}^{|\Lambda|}$ be the vector collecting the coefficients
of $u_\Lambda$ indexed in $\Lambda$; 
let $\mathbf{f}_\Lambda \in \mathbb{R}^{|\Lambda|}$ be the analogous restriction for
the vector of the coefficients of $f$. Finally, denote by $\mathbf{R}_\Lambda$ the
matrix that restricts a semi-infinite vector to the portion indexed in $\Lambda$, so that
$\mathbf{E}_\Lambda=\mathbf{R}_\Lambda^H$ is the corresponding extension matrix.
Then, setting
\begin{equation}\label{eq:four120}
\mathbf{A}_\Lambda = \mathbf{R}_\Lambda \mathbf{A} \mathbf{R}_\Lambda^H \;,
\end{equation}
 we preliminary observe that problem (\ref{eq:four.2}) can be equivalently written as
\begin{equation}\label{eq:four130}
\mathbf{A}_\Lambda \mathbf{u}_\Lambda = \mathbf{f}_\Lambda \;. 
\end{equation}  
%%%%%%
Next, let  $\mathbf{P}_\Lambda: \ell^2(\mathbb{N}_2) \to \ell^2(\mathbb{N}_2)$ be the projector operator defined as 
\[
(\mathbf{P}_\Lambda \mathbf{v})_\lambda=
\begin{cases}
v_\lambda & \text{\rm if } \lambda\in\Lambda \;, \\
0 & \text{\rm if } \lambda\notin\Lambda \;.
\end{cases}
\]
Note that $\mathbf{P}_\Lambda$ can be represented  as a diagonal semi-infinite matrix whose diagonal elements 
are $1$ for indexes belonging to $\Lambda$, \bnew{and} zero otherwise.  
Let us set $\mathbf{Q}_\Lambda=\mathbf{I}-\mathbf{P}_\Lambda$ and 
\bnew{introduce} the semi-infinite matrix $\widehat{\mathbf{A}}_\Lambda:=
\mathbf{P}_\Lambda \mathbf{A} \mathbf{P}_\Lambda + \mathbf{Q}_\Lambda$ which 
is equal to $\mathbf{A}_\Lambda$ for indexes in $\Lambda$ and to the identity matrix, otherwise. 
The definitions of the projectors $\mathbf{P}_\Lambda$ and $\mathbf{Q}_\Lambda$ easily yield the following result. 
\begin{property}[{invertibility of $\widehat\bA_\Lambda$}]\label{prop:inf-matrix}
If $\mathbf{A}$ is invertible with $\mathbf{A}\in\mathcal{D}_e(\eta_L)$, then the same holds for $\widehat{\mathbf{A}}_\Lambda$. \endproof
\end{property}

\noindent Finally, the infinite dimensional version of the Galerkin problem (\ref{eq:four.2}) reads as follows: 
find $\hat{\mathbf{u}}\in\ell^2(\mathbb{N}_2)$ such that  
\begin{equation}\label{eq:inf-pb-galerkin}
\widehat{\mathbf{A}}_\Lambda \hat{\mathbf{u}}
= \mathbf{P}_\Lambda \mathbf{f}\ .
\end{equation}
Let  $\bnew{{\mathbf{E}}_\Lambda={\mathbf{R}}_\Lambda^H}: \mathbb{R}^{\vert \Lambda\vert} \to \ell^2(\mathbb{N}_2)$ be the extension operator 
\bnew{and} let $\mathbf{u}_\Lambda\in \mathbb{R}^{\vert \Lambda\vert}$ be the Galerkin solution to  \eqref{eq:four130};
then, it is easy to check that $\hat{\mathbf{u}}={\mathbf{E}}_\Lambda \mathbf{u}_\Lambda$. 
%%%%%%%%%%%%%%%%%%%%%%%%%%%%%%%%v%%%%%%%%

We are now ready to state the main result of this section.
\begin{proposition}[{sparsity class of the residual}]\label{prop:unif-bound-res-exp}
{Let $\mathbf{A}\in\mathcal{D}_e(\eta_L)$ and
$\bA^{-1} \in\mathcal{D}_e(\bar\eta_L)$, for constants $\eta_L>0$
and $\bar\eta_L\in(0,\eta_L]$ according to Property
\ref{prop:inverse.matrix-estimate}. If $u \in {\mathcal A}^{\eta,t}_G$ for some $\eta>0$ and $t \in (0,1]$, such that $\eta < {\bar\eta_L}$, then
there exist suitable positive constants $\bar{\eta} \leq \eta$ and 
$\bar{t} \leq t$ such that
$r(u_\Lambda) \in {\mathcal A}_G^{\bar{\eta},\bar{t}}$ for any index
set $\Lambda$, with
}
$$
\Vert r(u_\Lambda) \Vert_{{\mathcal A}_G^{\bar{\eta},\bar{t}}} \lsim
\Vert u \Vert_{{\mathcal A}^{\eta,t}_G} \;.
$$
\end{proposition}
\begin{proof}
\bnew{This} is an adaptation of a similar \bnew{proof} in the periodic case given in \cite[Proposition 5.4]{CNV:2011}; we report here the details for completeness. Assume for the moment we are given $\bar{\eta}$ and $\bar{t}$. 
By using Proposition \ref{propos:spars-res}, 
\bnew{i.e. $\|\mathbf{A} \bf v\|_{\ell_G^{\bar\eta,\bar t}} 
\lesssim \|\bf v\|_{\ell_G^{\bar\eta_1,\bar t_1}}$}, and Property \ref{L:nlg2}, we get
\begin{equation}\label{eq:aux-2}
\begin{split}
\| {\bf r}_\Lambda\|_{\ell_G^{\bar{\eta},\bar{t}}(\mathbb{N}_2)} = \| \mathbf{A}({\bf u} - {\bf u}_\Lambda )\|_{\ell_G^{\bar{\eta},\bar{t}}(\mathbb{N}_2)} 
&\lesssim \| {\bf u} - {\bf u}_\Lambda \|_{\ell_G^{{\eta_1},{t_1}}(\mathbb{N}_2)}\\ 
&\lesssim {\| {\bf u} \|_{\ell_G^{2^{t_1}{\eta_1},{t_1}}(\mathbb{N}_2)} + 
\| {\bf u}_\Lambda \|_{\ell_G^{2^{t_1}{\eta_1},{t_1}}(\mathbb{N}_2)}},
\end{split}
 \end{equation}
 where $t_1$ and $\eta_1$ are defined by the relations 
 \[
 \bar t =\frac{t_1}{1+t_1}<t_1,\qquad  \bar\eta=\zeta(t_1)\eta_1 \ .
 \]
 From \eqref{eq:inf-pb-galerkin} we have $ {\bf u}_\Lambda = ({\widehat{\mathbf{A}}_\Lambda})^{-1} (\mathbf{P}_\Lambda{\bf f})$. 
Using Property \ref{prop:inf-matrix} and applying Proposition \ref{propos:spars-res} to $(\widehat{\mathbf{A}}_\Lambda)^{-1}$ we get 
\[
\|{\bf u}_\Lambda\|_{\ell_G^{2^{t_1}{\eta_1},{t_1}}(\mathbb{N}_2)}   = 
\|{\bf \widehat{u}}\|_{\ell_G^{2^{t_1}{\eta_1},{t_1}}(\mathbb{N}_2)} =
\|({\widehat{\mathbf{A}}_\Lambda})^{-1} (\mathbf{P}_\Lambda{\bf f})\|_{\ell_G^{2^{t_1}{\eta_1},{t_1}}(\mathbb{N}_2)}  \lesssim 
\| \mathbf{P}_\Lambda{\bf f} \|_{\ell_G^{{\eta_2},{t_2}}(\mathbb{N}_2)}  \leq \| {\bf f} \|_{\ell_G^{{\eta_2},{t_2}}(\mathbb{N}_2)}  \;, 
\]
with 
\[
 2^{t_1}\eta_1={\zeta(t_2)\eta_2 < \eta_2} \ , \qquad 
 {t_1}=\frac{t_2}{1+t_2}<t_2 \ .
 \]
By substituting the above inequality  into \eqref{eq:aux-2} and  using again Proposition \ref{propos:spars-res} we get 
\begin{equation}
\begin{split}
\| {\bf r}_\Lambda\|_{\ell_G^{\bar{\eta},\bar{t}}(\mathbb{N}_2)}
&\lesssim 
\| {\bf u} \|_{\ell_G^{2^{t_1}{\eta_1},{t_1}}(\mathbb{N}_2)} +
\| {\bf f} \|_{\ell_G^{{\eta_2},{t_2}}(\mathbb{N}_2)}\\ 
&= \| {\bf u} \|_{\ell_G^{2^{t_1}{\eta_1},{t_1}}(\mathbb{N}_2)} +
\| \mathbf{A}{\bf u} \|_{\ell_G^{{\eta_2},{t_2}}(\mathbb{N}_2)}
\lesssim  
\| {\bf u} \|_{\ell_G^{\eta,t}(\mathbb{N}_2)}
\end{split}
\end{equation}
where
\[
\eta_2={\zeta(t)\eta  <\eta}\ , \qquad 
 {t_2}=\frac{t}{1+t}<t \ .
 \]
This shows that the \bnew{assertion} holds true for the choice
\[
 \bar{\eta}=\Big(\frac12\Big)^{\frac{t}{1+2t}}
 \zeta \Big(\frac{t}{1+2t}\Big)
 \zeta \Big(\frac{t}{1+t}\Big)
 \zeta (t)
 \eta,
\qquad
\bar t = \frac{t}{1+3t}.
\]
It remains to verify the assumptions of Proposition
\ref{propos:spars-res} when $\bA$ is dense. We note that there holds  
\[
t_1 = \frac{t}{1+2t} < t_2 = \frac{t}{1+t} < t.
\]
Moreover, using $\eta_1<2^{t_1}\eta_1<\eta_2<\eta$ 
and $\eta_L \ge \bar\eta_L > \eta$ yields
\[
\eta< \eta_L, 
\qquad
\eta_1 < \eta_L,
\qquad
\eta_2 < \bar\eta_L,
\]
which are the required conditions to apply Proposition 
\ref{propos:spars-res} when $\bA$ is dense. This concludes the proof.
\end{proof}

%%%%%%%%%

%-----------------------------------------------------------------------------------
\section{The \bnew{predictor-corrector} adaptive algorithm}\label{subsec:aggress}
%-----------------------------------------------------------------------------------
In this section we study a variant (named {\bf PC-ADLEG}) of the ideal algorithm 
{\bf ADLEG}  introduced above.
Motivated by the enlightening discussion at the end of Section \ref{sec:plain-adapt-alg} 
and the subsequent results of Sections \ref{sec:nl} and
\ref{sec:spars-res}, we devise each iteration of the algorithm as
formed by a predictor step followed by a corrector step. The predictor
step guarantees an arbitrarily large error reduction (by suitably
enriching the output set from the D\"orfler procedure). This step is
driven by the sparsity class of the residual and  so, in view of
Proposition \ref{prop:unif-bound-res-exp} it does not guarantee
optimality with respect \bnew{to} the sparsity class of the exact solution.
The corrector step is realized by introducing a coarsening procedure which removes the smallest components of the output of the predictor step, in such a way to guarantee optimality.
\subsection{Enrichment}
We introduce the procedure {\bf ENRICH} defined as follows: 
\begin{itemize}
\item $\Lambda^* := \text{\bf ENRICH}(\Lambda,J)$ \\
Given an integer $J \geq 0$ and a finite set $\Lambda \subset \mathbb{N}_2$, the output is the set
$$
\Lambda^* := 
\{ k \in \mathbb{N}_2\ : \ \text{ there exists } \ell \in \Lambda \text{  such that } |k - \ell| \leq J \} \;.
$$
\end{itemize}
Note that since the procedure adds a $1$-dimensional ball of radius $J$ around each point of $\Lambda$, the cardinality
of the new set $\Lambda^*$ can be estimated as
\begin{equation}\label{eq:estim-enrich}
|\Lambda^*| \leq 2 J |\Lambda|\; .
\end{equation}

\bnew{Recall now that $\psi_{\bf A} (J,\eta)=C_{\bf A} {\rm e}^{-\eta_L J}$ 
from Property \eqref{prop:matrix-estimate} . Let}
$J_\theta>0$ be defined as 
\begin{equation}\label{eq:aggr2}
J_\theta:=\min\left\{ J\in\mathbb{N}:~\psi_{\mathbf{A}^{-1}} (J_\theta,\bar{\eta}_L)=C_{\mathbf{A}^{-1}} {\rm e}^{-\bar{\eta}_L J_\theta} \leq \sqrt{\frac{1-\theta^2}{\alpha_* \alpha^*}} ~\right\}\;.
\end{equation}
\begin{itemize}
\item $\Lambda^* := \text{\bf E-D\"ORFLER}(r, \theta)$\\
Given $\theta \in (0,1)$ and an element $r \in H^{-1}(I)$, 
the ouput $\Lambda^* \subset \mathbb{N}_2$ is defined by the sequence 
\begin{equation}\label{eq:aggr1}
\begin{split}
\widetilde{\Lambda}:=& \text{\bf D\"ORFLER}(r, \theta)\\
\Lambda^{*}:=& \text{\bf ENRICH}(\widetilde{\Lambda},J_\theta) \;.
\end{split}
\end{equation}
\end{itemize}
\subsection{Coarsening}\label{sec:coarsening}

We introduce the new procedure {\bf COARSE} defined as follows: 
\begin{itemize}
\item $\Lambda := {\bf COARSE}(w, \epsilon)$\\
Given a function $w \in V_{\Lambda^*}$ for some finite index set $\Lambda^*$, and an accuracy $\epsilon>0$
which is known to satisfy $\Vert u -  w \Vert \leq  \epsilon$,
 the output $\Lambda \subseteq \Lambda^*$ is a set of minimal cardinality such that
\begin{equation}\label{eq:def-coarse}
\Vert w - P_\Lambda w \Vert \leq 2 \epsilon \;.
\end{equation}
\end{itemize}

The following result shows that the cardinality $|\Lambda|$ 
is \rnew{optimal relative} to the sparsity class of $u$. We refer to
\rnew{Cohen \cite[Theorem 4.9.1]{Cohen-book:03} and Stevenson 
\cite[Proposition 3.2]{Stevenson:09}.}
\begin{theorem}[\rnew{cardinality} after coarsening]\label{T:coarsening}
Let $\vare>0$ and let $u \in {\mathcal A}^{\eta,t}_G$. %
Then 
\[
\vert \Lambda \vert  \le \frac{1}{\ \eta^{1/t}} 
\left(\log\frac{\Vert u \Vert_{{\mathcal A}^{\eta,t}_G}}{\vare}\right)^{1/t}\!\!\!+1 \;.
\]
\end{theorem}

The approximation error obtained after a call of {\bf COARSE} is estimated as follows.
\begin{property}[{error after coarsening}]\label{prop:cons-coarse}
The procedure {\bf COARSE} guarantees the bounds 
\begin{equation}\label{eq:def-coarse-bis}
\Vert u - P_\Lambda w \Vert \leq 3 \epsilon 
\end{equation}
and, for the Galerkin solution $u_\Lambda \in V_\Lambda$,
\begin{equation}\label{eq:def-coarse-ter}
\tvert u - u_\Lambda  \tvert \leq 3 \sqrt{\alpha^*}\epsilon \;.
\end{equation}
\end{property}
\proof
The first bound is trivial, the second one follows from the minimality property of the Galerkin solution 
in the energy norm and from (\ref{eq:four.1bis}):
$$
\tvert u - u_\Lambda \tvert \leq \tvert u - P_{\Lambda}w \tvert \leq \sqrt{\alpha^*}
\Vert u-P_{\Lambda} w \Vert \leq 3 \sqrt{\alpha^*} \epsilon \;. \qquad \quad \endproof
$$

\subsection{PC-ADLEG: a predictor-corrector version of ADLEG}\label{sec:moddorfler}

Given two parameters $\theta \in (0,1)$ and $tol \in [0,1)$, we define the following adaptive algorithm. 

\medskip
{\bf Algorithm PC-ADLEG}($\theta, \ tol$)
\begin{itemize}
\item[\ ] Set $r_0:=f$, $\Lambda_0:=\emptyset$, $n=-1$
\item[\ ] do
	\begin{itemize}
	\item[\ ] $n \leftarrow n+1$
	\item[\ ] $\widehat{\partial\Lambda}_{n}:= \text{\bf E-D\"ORFLER}(r_{n}, \theta)$		
        \item[\ ] $\widehat\Lambda_{n+1}:=
			\Lambda_{n} \cup \widehat{\partial\Lambda}_{n}$
	\item[\ ] $\widehat{u}_{n+1}:= {\bf GAL}(\widehat\Lambda_{n})$
	%\item[\ ] $\widehat{r}_{n+1}:= {\bf RES}(\widehat{u}_{n+1})$
	\item[\ ] $\Lambda_{n+1}:=
		{\bf COARSE}\left(\widehat{u}_{n+1}, \bnew{\frac{2}{\alpha_*}
		\sqrt{1-\theta^2}\|r_n\|}\right)$
	\item[\ ] $u_{n+1}:={\bf GAL}(\Lambda_{n+1})$
	\item[\ ] $r_{n+1}:={\bf RES}(u_{n+1})$
	\end{itemize} 
\item[\ ]  while $\Vert r_{n+1} \Vert > tol $
\end{itemize}

\begin{theorem}[{contraction property of {\bf PC-ADLEG}}]\label{teo:four2}
Let \bnew{$0<\theta<1$ be chosen so that}
\begin{equation}\label{eq:def_rhotheta}
\bnew{\rho=\rho(\theta)= 6\frac{{\alpha^*}}{\alpha_*} \sqrt{1-\theta^2}<1\;.}
\end{equation}
{If} the assumptions of Property \ref{prop:inverse.matrix-estimate} are fulfilled, the sequence of errors $u-u_n$ generated for $n\geq 0$ by the algorithm satisfies the inequality
$$
  \tvert u-u_{n+1} \tvert \leq \rho \tvert u-u_n \tvert \;.
$$ 
Thus, \bnew{for} any $tol>0$ the algorithm terminates in a finite number of iterations, whereas for $tol=0$
the sequence $u_n$ converges to $u$ in $H^1(I)$ as $n \to \infty$.
\end{theorem}
\begin{proof}
At the $n$-th \bnew{step, we have
$\widehat{\partial\Lambda}_n = \text{\bf E-D\"ORFLER}(r_n\theta)$ and
$\widehat{\Lambda_{n+1}}=\Lambda_{n} \cup \partial\Lambda_{n}$, where} 
\begin{equation}\label{eq:aggr1}
\begin{split}
\widetilde{\partial\Lambda}_{n} =& ~\text{\bf D\"ORFLER}(r_n, \theta)\\
\bnew{\widehat{\partial\Lambda}_{n}} =& ~\text{\bf ENRICH}(\widetilde{\partial\Lambda}_{n},J_\theta) \;.
\end{split}
\end{equation}
We recall that the set $\widetilde{\partial\Lambda}_n$ is such that 
$g_n= P_{\widetilde{\partial\Lambda}_n} r_n$ satisfies 
$$
\Vert r_n- g_n \Vert \leq \sqrt{1-\theta^2}  \Vert r_n \Vert 
$$
(see (\ref{eq:four.2.5.5bis})). Let $w_n \in V$ be the solution of $L w_n = g_n$, which in general
will have infinitely many components, and let us split it as
$$
w_n= \bnew{P_{\widehat\Lambda_{n+1}} w_n +
  P_{\widehat\Lambda_{n+1}^c} w_n =: y_n + z_n \in V_{\widehat\Lambda_{n+1}} \oplus 
 V_{\widehat\Lambda_{n+1}^c} \;.}
$$
Then, by the minimality property of the Galerkin solution 
\bnew{$\widehat u_{n+1}\in V_{\widehat \Lambda_{n+1}}$} in the energy norm, and by (\ref{eq:four.1bis})
and (\ref{eq:four.2.1bis}), one has
\begin{eqnarray*}
\tvert u - \widehat{u}_{n+1} \tvert &\leq&  \tvert u - (u_{n}+y_{n}) \tvert \leq  
 \tvert u- u_n - w_{n} + z_{n} \tvert \\
&\leq& \frac1{\sqrt{\alpha_*}} \Vert L(u- u_n - w_{n}) \Vert + \sqrt{\alpha^*}\Vert z_{n} \Vert
= \frac1{\sqrt{\alpha_*}} \Vert r_n - g_{n} \Vert  + \sqrt{\alpha^*} \Vert z_{n} \Vert \;.
\end{eqnarray*}
Thus,
$$
\tvert u - \widehat{u}_{n+1} \tvert \leq \frac{1}{\sqrt{\alpha_*}}\sqrt{(1-\theta^2)} \, \Vert r_n \Vert 
+ \sqrt{\alpha^*} \Vert z_{n}\Vert \;.
$$
\bnew{Since $z_n= \big( P_{\widehat\Lambda_{n+1}^c} L^{-1} P_{\widetilde{\partial\Lambda}_n} \big) r_n $}, observing that 

$$ k \in \bnew{\widehat\Lambda_{n+1}^c} \quad \text{and} \quad \ell \in \widetilde{\partial\Lambda}_n\qquad \Rightarrow \qquad
|k - \ell | > J_\theta \;,
$$
we have 
$$
\Vert \bnew{P_{\widehat\Lambda_{n+1}^c}} L^{-1} P_{\widetilde{\partial\Lambda}_n} \Vert \leq
\Vert \mathbf{A}^{-1}-(\mathbf{A}^{-1})_{J_\theta} \Vert \leq 
\psi_{\mathbf{A}^{-1}} (J_\theta,\bar{\eta}_L) \leq  \sqrt{\frac{1-\theta^2}{\alpha_* \alpha^*}}\;,
$$ 
where we have used \eqref{eq:aggr2}. Thus, we obtain
\begin{equation}\label{eq:aggr_error_reduct}
\tvert u - \widehat{u}_{n+1} \tvert \leq \bnew{\frac{2}{\sqrt{\alpha_*}} \sqrt{1-\theta^2} \, \Vert r_n\Vert} 
\end{equation}
or, equivalently, 
\begin{equation}\label{eq:aggr_error_reduct}
\| u - \widehat{u}_{n+1} \| \leq \bnew{\frac{2}{{\alpha_*}} \sqrt{1-\theta^2} \, \Vert r_n \Vert\;.}
\end{equation}
Since the right-hand side \bnew{of \eqref{eq:aggr_error_reduct}}
is precisely the parameter $\epsilon_n$
fed to the procedure {\bf COARSE}, \bnew{Property \eqref{prop:cons-coarse} implies}
\begin{equation}\label{aux:opt}
\tvert u- u_{n+1} \tvert \leq \bnew{6\frac{\sqrt{\alpha^*}}{\alpha_*} \sqrt{1-\theta^2} \| r_n\| 
\leq 6\frac{{\alpha^*}}{\alpha_*} \sqrt{1-\theta^2} \tvert u-u_n\tvert   \;,}
\end{equation}
\bnew{for the Galerkin solution $u_{n+1}\in V_{\Lambda_{n+1}}$. The
  assertion thus follows immediately.}
\end{proof}

The rest of the paper will be devoted to investigating complexity issues for the sequence of approximations 
$u_n=u_{\Lambda_n}$ generated by {\bf PC-ADLEG}. In particular, we wish to estimate the cardinality
of each $\Lambda_n$ and check whether its growth is ``optimal'' with respect to the sparsity class 
${\mathcal A}_\phi$ of the exact solution, in the sense that $|\Lambda_n|$ is comparable to 
the cardinality of the index set of the best approximation of $u$ yielding the same error
$\Vert u - u_n \Vert$.

\begin{theorem}[{cardinality of {\bf PC-ADLEG}}]\label{teo:pc-ADLEG}
Suppose that $u \in {\mathcal A}^{\eta,t}_G$, for some $\eta >0$ and $t \in (0,1]$. 
Then, there exists a constant $C>1$ such that the cardinality of the set 
$\Lambda_n$ of the active degrees of freedom produced by {\bf PC-ADLEG} satisfies the bound 
$$
|\Lambda_n| \leq \frac{1}{\eta^{1/t}}
\left( \log  \frac{\Vert u \Vert_{{\mathcal A}^{\eta,t}_G}}{\Vert u-u_n \Vert} + \log C \right)^{1/t} +1\;, 
{\qquad\forall\ n\ge0.}
$$
If, in addition, the assumptions of Proposition \ref{prop:unif-bound-res-exp} are satisfied, then 
the cardinality of the intermediate sets $\widehat\Lambda_{n+1}$ activated in the predictor step
can be estimated as
$$
|\widehat\Lambda_{n+1}| \leq | \Lambda_n | +
 \frac{2 J_\theta}{\bar{\eta}^{1/\bar{t}}}
\left( \log  \frac{\Vert u \Vert_{{\mathcal A}^{\eta,t}_G}}{\Vert u-u_{n+1} \Vert}  + \log C \right)^{1/\bar{t}} + 2J_\theta\;, 
{\qquad\forall\ n\ge0\;,}
$$
where $J_\theta$ is defined in \eqref{eq:aggr2} and $\bar{\eta}\leq \eta$, $\bar{t}\leq t$ 
are the parameters which occur in the thesis of Proposition \ref{prop:unif-bound-res-exp}.
\end{theorem}
\begin{proof} 
\bnew{Combining} Theorem \ref{T:coarsening} with 
$\epsilon_{n}= \bnew{\frac{2}{\alpha_*}\sqrt{1-\theta^2}\|r_n\|}$
\bnew{and Property \ref{prop:cons-coarse}}, we get
\begin{eqnarray}
\vert \Lambda_{n+1} \vert  &\le& \frac{1}{\ \eta^{1/t}} 
\left(\log\frac{\Vert u \Vert_{{\mathcal A}^{\eta,t}_G}}{\epsilon_n}\right)^{1/t}\!\!\!+1
\leq\frac{1}{\eta^{1/t}}
\left( \log  \frac{\Vert u \Vert_{{\mathcal A}^{\eta,t}_G}}{\Vert u-u_{n+1} \Vert} + \log C \right)^{1/t} + 1\ .\nonumber
\end{eqnarray}
\bnew{We now estimate $\vert \widehat{\Lambda}_{n+1} \vert$.}
Recalling \eqref{eq:estim-enrich}, we have 
$\vert \bnew{\widehat{\partial\Lambda}_n}\vert \leq  2J_\theta \vert
\widetilde{\partial\Lambda}_n\vert $,
\bnew{ and using \eqref{aux:residual} together
 with Proposition \ref{propos:spars-res} and \eqref{aux:opt} yields} 
\begin{eqnarray}
\vert \widetilde{\partial\Lambda}_n\vert 
&\leq & \frac{1}{ \bar\eta^{1/\bar{t}}} \left(\log \frac{ \| r_n \|_{{\mathcal A}^{\bar\eta,t}_G}}{\sqrt{1-\theta^2}\|r_n\|}\right)^{1/\bar{t}}+1\nonumber\\
&\leq&   \frac{1}{ \bar\eta^{1/\bar{t}}} \left(\log \frac{ \| u
    \|_{{\mathcal A}^{\eta,t}_G}}{\|u-u_{n+1}\|} + \log C
\right)^{1/\bar{t}}+1\nonumber.
\end{eqnarray}
\bnew{Finally, the second assertion follows from $\vert \widehat{\Lambda}_{n+1} \vert \leq \vert \Lambda_n\vert + \vert \partial\Lambda_n\vert $.}
 \end{proof}

We observe that the cardinality of $\Lambda_n$, i.e., the set of degrees of freedom of the Galerkin solution at the end of each iteration, is optimal. On the other hand, in the case $\bar\eta<\eta$ and $\bar t< t$, 
the cardinality  of $\widehat\Lambda_{n+1}$ may grow at a faster rate than the cardinality of 
$\Lambda_n$. This is a direct consequence of the fact that, according
to Proposition \ref{prop:unif-bound-res-exp}, the sparsity class of
the residual may be worse than \rnew{that of the solution.
Moreover, the contraction constant in
\eqref{eq:def_rhotheta} is as close to $0$ as desired provided
$\theta$ is close to $1$, which is
consistent with the expected fast error decay of spectral methods.
This is a key feature of our contribution and a novel idea with
respect to the standard algebraic case; see the surveys 
\cite{NSV:09,Stevenson:09}.

Similar results to Theorems \ref{teo:four2} and 
\ref{teo:pc-ADLEG} are valid also for the
algebraic case, with optimal convergence rates only limited by solution 
regularity; they can be derived as in \cite[Theorems 3.3 and 7.2]{CNV:2011}. 
In addition, a more
conservative version of {\bf PC-ADLEG}, which avoids the procedure
{\bf ENRICH}, t
can be studied as well following  \cite[Theorem 8.1]{CNV:2011}. This
version exhibits better optimality properties than {\bf PC-ADLEG}
at the expense of a relatively large contraction factor, which is at
odds with spectral accuracy.}

%-----------------------------------------------------------------------------------
\section*{Acknowledgements}
We wish to thank Luigi Rodino for insightful discussions on Gevrey spaces.
%-----------------------------------------------------------------------------------

%%%%%%%%%%%%%%%%%%%%%%%%%%%%%%%%%%%%%%%%%%%%%%%%%%%%%%%%%%%%%%%%%%%%%%%%%%%%%%%%%%%%

\begin{thebibliography}{10}

\bibitem{Adams:1878}
J.C. Adams.
\newblock On the expression of the product of any two legendreÕs coefficients
  by means of a series of legendreÕs coefficients.
\newblock {\em Proceedings of the Royal Society of London}, 27:63--71, 1878.

\bibitem{BG:72}
M.~S. Baouendi and C.~Goulaouic.
\newblock R\'egularit\'e analytique et it\'er\'es d'op\'erateurs elliptiques
  d\'eg\'en\'er\'es; applications.
\newblock {\em J. Functional Analysis}, 9:208--248, 1972.

\bibitem{BDD:04}
P.~Binev, W.~Dahmen, and R.~DeVore.
\newblock Adaptive finite element methods with convergence rates.
\newblock {\em Numer. Math.}, 97(2):219--268, 2004.

\bibitem{Burg-Doerfler:11}
M.~B{\"u}rg and W.~D{\"o}rfler.
\newblock Convergence of an adaptive {$hp$} finite element strategy in higher
  space-dimensions.
\newblock {\em Appl. Numer. Math.}, 61(11):1132--1146, 2011.

\bibitem{CNV:2011}
C.~Canuto, Ricardo~H. Nochetto, and M.~Verani.
\newblock {A}daptive {F}ourier-{G}alerkin {M}ethods.
\newblock {\em arXiv:1201.5648, submitted}, pages 1--48, 2011.

\bibitem{Nochetto-et-al:2008}
J.~M. Cascon, C.~Kreuzer, R.~H. Nochetto, and K.~G. Siebert.
\newblock Quasi-optimal convergence rate for an adaptive finite element method.
\newblock {\em SIAM J. Numer. Anal.}, 46(5):2524--2550, 2008.

\bibitem{CDDV:1998}
A.~Cohen, W.~Dahmen, and R.~DeVore.
\newblock Adaptive wavelet methods for elliptic operator equations --
  convergence rates.
\newblock {\em Math. Comp}, 70:27--75, 1998.

\bibitem{CDN:11}
A.~Cohen, R.~DeVore, and R.H. Nochetto.
\newblock Convergence rates for afem with ${H}^{-1}$ data.
\newblock {\em Found. Comput. Math., to appear}, 2011.

\bibitem{Cohen-book:03}
Albert Cohen.
\newblock {\em Numerical analysis of wavelet methods}, volume~32 of {\em
  Studies in Mathematics and its Applications}.
\newblock North-Holland Publishing Co., Amsterdam, 2003.

\bibitem{Dahlke-Fornasier-Groechenig:2010}
S.~Dahlke, M.~Fornasier, and K.~Groechenig.
\newblock Optimal adaptive computations in the {J}affard algebra and localized
  frames.
\newblock {\em Journal of Approximation Theory}, 162(1):153 -- 185, 201.

\bibitem{Davis:1963}
Philip~J. Davis.
\newblock {\em Interpolation and approximation}.
\newblock Blaisdell Publishing Co. New York-Toronto-London, 1963.

\bibitem{dorfler:96}
W.~D{\"o}rfler.
\newblock A convergent adaptive algorithm for {P}oisson's equation.
\newblock {\em SIAM J. Numer. Anal.}, 33(3):1106--1124, 1996.

\bibitem{Doerfler-Heuveline:07}
W.~D{\"o}rfler and V.~Heuveline.
\newblock Convergence of an adaptive {$hp$} finite element strategy in one
  space dimension.
\newblock {\em Appl. Numer. Math.}, 57(10):1108--1124, 2007.

\bibitem{Jaffard:1990}
S.~Jaffard.
\newblock Propri\'et\'es des matrices "bien localis\'ees" pr\`es de leur
  diagonale et quelques applications.
\newblock {\em Annales de l'I.H.P.}, 5:461--476, 1990.

\bibitem{MNS:00}
P.~Morin, R.~H. Nochetto, and K.~G. Siebert.
\newblock Data oscillation and convergence of adaptive {FEM}.
\newblock {\em SIAM J. Numer. Anal.}, 38(2):466--488 (electronic), 2000.

\bibitem{NSV:09}
R.~H. Nochetto, K.~G. Siebert, and A.~Veeser.
\newblock Theory of adaptive finite element methods: an introduction.
\newblock In {\em Multiscale, nonlinear and adaptive approximation}, pages
  409--542. Springer, Berlin, 2009.

\bibitem{Rachowicz-etal:06}
W.~Rachowicz, D.~Pardo, and L.~Demkowicz.
\newblock Fully automatic hp-adaptivity in three dimensions.
\newblock {\em Computer Methods in Applied Mechanics and Engineering},
  195(37-40):4816--4842, 2006.

\bibitem{Schmidt-Siebert:00}
Alfred Schmidt and Kunibert~G. Siebert.
\newblock A posteriori estimators for the {$h$}-{$p$} version of the finite
  element method in 1{D}.
\newblock {\em Appl. Numer. Math.}, 35(1):43--66, 2000.

\bibitem{Stevenson:2007}
R.~Stevenson.
\newblock Optimality of a standard adaptive finite element method.
\newblock {\em Found. Comput. Math.}, 7(2):245--269, 2007.

\bibitem{Stevenson:09}
Rob Stevenson.
\newblock Adaptive wavelet methods for solving operator equations: an overview.
\newblock In {\em Multiscale, nonlinear and adaptive approximation}, pages
  543--597. Springer, Berlin, 2009.

\end{thebibliography}
\end{document}